\documentclass[10pt]{article}
\pdfoutput=1
\usepackage{amsmath,amssymb,amsfonts,amsthm} 
\usepackage{mathtools}
\usepackage{stmaryrd} 
\usepackage{color}
\usepackage{calc}
\usepackage[hidelinks]{hyperref}
\usepackage[version=3,arrows=pgf-filled]{mhchem}
\usepackage{csquotes}
\usepackage{colortbl}
\usepackage{extarrows}
\usepackage[ruled]{algorithm2e} 
\usepackage{mathrsfs} 
\usepackage{import} 
\usepackage{enumitem}
\usepackage{appendix}
\usepackage{fullpage}
\usepackage{pgfplots}
\usepackage{thmtools}   

\usepackage{tikz}
\usepackage{pgfplots}
\pgfplotsset{compat=1.14}
\usetikzlibrary{spy,calc,calc,decorations.markings,math,arrows.meta,shapes.misc,fillbetween,arrows,decorations.pathmorphing,backgrounds,positioning,fit,matrix}

\usepackage{dsfont}
\def\Indicator{\mathds{1}}

\mathchardef\mhyphen="2D 

\theoremstyle{definition}
\newtheorem{defi}{Definition}[section]

\theoremstyle{remark}
\newtheorem{rem}[defi]{Remark}
\usepackage{etoolbox} 
\AtEndEnvironment{rem}{\qed}
\AtEndEnvironment{example}{\qed}
\newtheorem{example}[defi]{Example}

\theoremstyle{plain}
\newtheorem{theorem}[defi]{Theorem}
\newtheorem{lem}[defi]{Lemma}
\newtheorem{cor}[defi]{Corollary}
\newtheorem{prop}[defi]{Proposition}

\newcommand{\N}{\ensuremath{\mathbb{N}}}
\newcommand{\R}{\ensuremath{\mathbb{R}}}

\newcommand{\norm}[2][]{\ensuremath{\left\|#2\right\|_{#1}}} 

\newcommand{\snorm}[2][]{\ensuremath{|#2|_{#1}}}
\newcommand{\snormlr}[2][]{\ensuremath{\left|#2\right|_{#1}}}
\newcommand{\scalarprod}[2]{\ensuremath{\left\langle{#1,#2}\right\rangle}} 

\newcommand{\curlH}{\ensuremath{\mathcal{H}}}
\newcommand{\curlF}{\ensuremath{\mathcal{F}}}
\newcommand{\curlX}{\ensuremath{\mathcal{X}}}
\newcommand{\curlM}{\ensuremath{\mathcal{M}}}
\newcommand{\curlP}{\ensuremath{\mathcal{P}}}
\newcommand{\curlL}{\ensuremath{\mathcal{L}}}
\newcommand{\curlA}{\ensuremath{\mathcal{A}}}
\newcommand{\curlE}{\ensuremath{\mathcal{E}}}
\newcommand{\curlY}{\ensuremath{\mathcal{Y}}}
\newcommand{\curlZ}{\ensuremath{\mathcal{Z}}}
\newcommand{\curlJ}{\ensuremath{\mathcal{J}}}
\newcommand{\curlI}{\ensuremath{\mathcal{I}}}
\newcommand{\curlC}{\ensuremath{\mathcal{C}}}

\newcommand{\supp}{\ensuremath{\operatorname{supp}}}

\newcommand{\gammalim}{\ensuremath{\operatorname*{\Gamma-lim}}}

\newcommand{\e}{\varepsilon}

\newlength{\leftstackrelawd}
\newlength{\leftstackrelbwd}
\def\leftstackrel#1#2{\settowidth{\leftstackrelawd}%
	{${{}^{#1}}$}\settowidth{\leftstackrelbwd}{$#2$}%
	\addtolength{\leftstackrelawd}{-\leftstackrelbwd}%
	\leavevmode\ifthenelse{\lengthtest{\leftstackrelawd>0pt}}%
	{\kern-.5\leftstackrelawd}{}\mathrel{\mathop{#2}\limits^{#1}}}

\def\RelEnt{\mathscr{H}}
\def\RF{\mathscr{R}}
\def\I{\mathscr{I}}
\def\LDJ{\mathcal{J}}

\def\Cz{{c_0}}
\def\Cb{{\ell^\infty}}

\newcommand{\calF}{\mathscr{F}}
\newcommand{\calL}{\mathscr{L}}

\title{An inequality connecting entropy distance, Fisher Information and large deviations}
\author{Bastian Hilder, Mark A.\ Peletier, Upanshu Sharma, Oliver Tse}
\date{\today}

\begin{document}
\maketitle
	
\begin{abstract}
In this paper we introduce a new generalisation of the relative Fisher Information for Markov jump processes on a finite or countable state space, and prove an inequality which connects this object with the  relative entropy and a large deviation rate functional. In addition to possessing various favourable properties, we show that this \emph{generalised Fisher Information} converges to the classical Fisher Information in an appropriate limit.  We then use this generalised Fisher Information and the aforementioned inequality to qualitatively study coarse-graining problems for jump processes on discrete spaces. 
\end{abstract}
	


\section{Introduction}\label{sec:introduction}	
Lyapunov functions are important tools in the study of evolution equations. The relative entropy, which for two probability measures $\mu,\rho\in\curlP(\curlX)$ is given by
\begin{align}\label{DFIR-def:Intro-RelEnt}
\RelEnt(\mu \vert \rho) = \begin{dcases}
\int_\curlX f \log f \,d\rho , & \text{if } f = \dfrac{d\mu }{d\rho} \text{ exists}, \\
+ \infty, & \text{otherwise,}
\end{dcases}
\end{align}
is one such Lyapunov function that plays a crucial role in the study of forward Kolmogorov equations. 
These equations describe the evolution of the distribution of a Markov process. In recent years, extensive research has been devoted to the  study of the relative entropy and the Fisher Information (entropy production) which, amongst other things, are used to study the trend to equilibrium for both continuous~\cite{ArnoldCarrilloDesvillettesDolbeaultJuengelLedermanMarkowichToscaniVillani04,MichelMischlerPerthame05} and discrete state-space Markov processes~\cite{diaconisSaloffCoste96,bobkovTetali06}. 
Typically this involves studying  the time evolution of the relative entropy~\eqref{DFIR-def:Intro-RelEnt} where $\rho$ is the stationary solution and $\mu_t$ is the time-dependent solution of the forward Kolmogorov equation under consideration. 
Although it is not a metric on the space of probability measures, relative entropy has been used as a notion of distance to equilibrium due to its favourable properties and  natural connections to statistical physics.

As opposed to what was described above, in certain cases the relative entropy is also used to compare the time-dependent distributions of two different Markov processes.
In the context of hydrodynamic limits, Yau~\cite{yau91} uses the relative entropy to compare the evolution of finite particle evolution with certain local-Gibbs states. 
Legoll and Leli{\'e}vre~\cite{legollLelievre10} use relative entropy to compare an approximate solution with the true solution of a Fokker-Planck equation arising in molecular dynamics, and Bogachev et al.\ \cite{bogachevRoecknerShaposhnikov16} compare solutions of two different Fokker-Planck equations in the context of mean-field games.

It has recently been shown~\cite{DLPSS-TMP} that the relative entropy comparing an arbitrary time-dependent probability measure to the solution of a Fokker-Planck equation is directly linked to the Fisher Information and the large-deviation rate functional via an inequality.
We refer to~\cite[Chapter 2]{Sharma17} for a detailed overview.
In~\cite{duongLamaczPeletierSharma17} the authors present a new variational approach that uses this inequality to qualitatively study coarse-graining problems in (nonlocal) Fokker-Planck equations. 
In~\cite{DLPSS-TMP} this inequality has been used to quantitatively estimate coarse-graining errors. 

While all the aforementioned references deal with diffusion processes, not much is known about the the relative entropy of two time-dependent distributions for jump processes.
In recent years, for processes on discrete spaces, new Wasserstein-like gradient-flow structures with relative entropy as  the driving functional have been discovered~\cite{maas11,mielke11,mielke13,chowHuangLiZhou,erbarFathiLaschosSchlichting16}. 
In this paper we ask if the ideas described above for the continuous case can be generalised to  the discrete case, specifically for Markov jump processes:
\begin{center}
	Starting with Markov jump processes, can the relative entropy of two time-dependent curves be connected to the large-deviation rate functional? Furthermore, can this connection be exploited to study coarse-graining problems?
\end{center}	
In this paper we provide an answer to these questions by generalising the notion of Fisher Information for Markov processes. In addition to studying its properties, we will show that this generalised Fisher Information is naturally related to the relative entropy and the large-deviation rate functional. Finally we apply this inequality to study a coarse-graining problem on a discrete state space. 

\subsection{Relative Fisher Information and large-deviation rate functional}\label{sec-RFI-LDP}
Before we present our contributions to answering the questions mentioned above (see Section~\ref{sec:MainRes}), we introduce the classical relative Fisher Information and the large-deviation rate functional. Unlike the relative entropy, these two objects explicitly depend on the evolution equation under consideration. 

In this paper we are interested in jump processes on a \emph{finite or countable state space}~$\curlX$. The law of the process $\rho : [0,T] \rightarrow \curlP(\curlX)$ satisfies the evolution equation
\begin{align}
\begin{dcases}
\partial_t \rho = L^T \rho, \\
\rho_{t = 0} = \rho_0,
\end{dcases}
\label{eq:forwardKolmogorovEquation}
\end{align}
in the space of probability measures $\curlP(\curlX)$. 
In equation~\eqref{eq:forwardKolmogorovEquation},  $L^T$ is the adjoint of $L: \Cz(\curlX) \rightarrow \Cz(\curlX)$, the generator of the process.
Since $\curlX$ is discrete, we use matrix notation and  write the operator $L$ as a (potentially infinite) matrix $L\in\R^{\curlX\times\curlX}$. The generator $L$ satisfies 
\begin{subequations}
	\label{eq:conditionsGenerator}
	\begin{align}
	&(L1)
	\qquad \ L(x,y) \geq 0 \text{ for all } x \neq y \text{ and } 
	\sum_{y \in \curlX} L(x,y) = 0 \text{ for all } x \in \curlX, 
	\label{eq:conditionsGenerator1}\\
	&(L2) \qquad \sup_{x\in \curlX} |L(x,x)|<\infty,
	\label{eq:conditionsGenerator2} \\
	&(L3) \qquad L \text{ is irreducible}.
	\label{eq:conditionsGenerator3}
	\end{align}
\end{subequations}
These conditions are sufficient for $L$ to be a bounded Markov operator $L:\Cz(\curlX)\to \Cz(\curlX)$, where $c_0(\curlX)$ is the Banach space of functions on $\curlX$ that converge to zero outside of large compact subsets of $\curlX$, equipped with the supremum norm.  Since $L^T$  generates a uniformly continuous semigroup in $\ell^1(\curlX)$~\cite[Proposition 2.11]{engelNagel06}, equation~\eqref{eq:forwardKolmogorovEquation}  admits a unique solution $\rho\in \curlC^1([0,T];\ell^1(\curlX))$~\cite[Theorem 6.6]{engelNagel06}; since equation~\eqref{eq:forwardKolmogorovEquation} preserves non-negativity and total mass, we have $\rho\in \curlC([0,T],\curlP(\curlX))$ whenever $\rho_0\in \curlP(\curlX)$. 

\begin{rem}\label{rem:weak-strong-equivalence}
	The space $\curlP(\curlX)$ is a subset of $\ell^1(\curlX)$, and the weak measure topology on $\curlP(\curlX)$ coincides with the $\sigma(\ell^1,\ell^\infty)$-topology on $\ell^1(\curlX)$. 
	Recall that by Schur's theorem, weak and strong \emph{convergence} on $\ell^1(\curlX)$ are the same, even though the weak and strong topologies may be different; therefore functions $f:[0,T]\to\ell^1(\curlX)$ are strongly continuous if and only they are weakly continuous. Since `weak measure convergence' in $\curlP(\curlX)$ is the same as the $\sigma(\ell^1,\ell^\infty)$-convergence in $\ell^1(\curlX)$, we will omit the term `weak' in our discussion and notation, and simply talk about `continuous' functions from $[0,T]$ to $\curlP(\curlX)$ or to $\ell^1(\curlX)$.
\end{rem}

The classical definition  of `relative Fisher Information' arises from the time derivative of the relative entropy along two solutions of~\eqref{eq:forwardKolmogorovEquation}. Indeed, for two positive solutions $\mu,\rho$ of~\eqref{eq:forwardKolmogorovEquation}, we have 
\begin{align}\label{eq:Ent-Fish-Rel}
\dfrac{d}{dt} \RelEnt(\mu_t \vert \rho_t) = -\RF_L(\mu_t \vert \rho_t), 	
\end{align}
where $\mu_t, \rho_t$ denote the time slice at time $t$, and the right-hand side is defined as follows.

\begin{defi}
	For $\mu,\rho\in \curlP_+(\curlX)$, the set of strictly positive probability measures, the (classical) \emph{relative Fisher Information} is defined as
	\begin{equation}
	\label{eq:RFIBasicDef}
	\RF_L(\mu \vert \rho) 
	:=  \sum_{x,y \in \curlX} \rho(x) L(x,y) \left[v(y) - v(x) - v(x)\log\left(\frac{v(y)}{v(x)}\right)\right], \qquad v=\mu/\rho
	\end{equation}
\end{defi}

This sum is well-defined in $[0,\infty]$, since $L(x,y)\geq0$ for $x\not=y$, and the term between brackets is non-negative and vanishes if $x=y$.
Especially, the relative Fisher Information is non-negative.
This corresponds to the well-known fact that the relative entropy decays in time along two solutions of the same forward Kolmogorov equation (see~\cite[Theorem 1.1]{voigt81}).
It should be noted that the definition~\eqref{eq:RFIBasicDef} of the Fisher Information coincides with the classical notion of Fisher Information with respect to the stationary measure, i.e.\ when $L^T\rho=0$ (see~\cite[Equation 1.4]{bobkovTetali06}). Alternatively, the relative Fisher Information~\eqref{eq:RFIBasicDef} can also be seen as the Bregman divergence of the Fisher Information with respect to the stationary measure (see~\cite[Section 5.1]{hilder17} for details).  

Apart from the classical connection between (linear) Markov processes and forward Kolmogorov equations described above, the forward Kolmogorov equations can also be viewed as the many-particle limit of some underlying system of Markov processes. To make this precise, consider a  sequence $(X^n)_{n \in \N}$ of independent and identical Markov processes on state space $\curlX$ and generated by $L$. Under fairly general conditions (see for instance~\cite[Theorem 11.4.1]{dudley89}), the sequence of empirical measures
\begin{align}
\rho^N := \dfrac{1}{N} \sum_{i = 1}^N \delta_{X^i},
\label{eq:empiricalProcesses}
\end{align}
converges almost surely to the solution of~\eqref{eq:forwardKolmogorovEquation}. 

This convergence is the starting point for a large-deviation result. In particular it has been shown (see Theorem~\ref{thm:LDP} below) that the sequence $\rho^N$ has a \emph{large-deviation property} which characterises the probability of finding the empirical measure far from the limit $\rho$, written informally as  
\begin{align*}
\operatorname{Prob}(\rho^N \approx \rho) \sim e^{-N(\I_0(\rho_0) + \I_L(\rho))} \  \text{ as } N \rightarrow \infty,
\end{align*}
in terms of \emph{rate functionals} $\I_0$ and $\I_L$ of the initial data $(\rho^N_0)_{N \in \N}$ and the path $(t \mapsto \rho^N_t)_{N \in \N}$ respectively.
In this paper we will focus on $\I_L:\curlC([0,T];\curlP(\curlX))\rightarrow[0,\infty]$ which is given by
\begin{align}
\I_L(\rho) = \begin{dcases}
\int_0^T \curlL(\rho_t,\partial_t \rho_t) \,dt, & \text{ if } \rho \in A\curlC([0,T];\curlP(\curlX)), \\
+ \infty, & \text{ otherwise.}
\end{dcases}
\label{eq:rateFunctional}
\end{align}
Here $A\curlC([0,T];\curlP(\curlX))$ is the space of absolutely continuous trajectories in the space of probability measures (see Appendix~\ref{app:Bochner}).

The \emph{Lagrangian} $\curlL:\curlP(\curlX)\times \ell^1(\curlX)\rightarrow[0,\infty]$ in the definition above of $\I_L$ is non-negative and convex in its second argument, and satisfies $\curlL(\rho_t, \partial_t \rho_t) = 0$ if and only if  $\rho$ solves $\partial_t \rho_t = L^T \rho_t$. The rate functional~$\I_L$ therefore has the crucial properties
\begin{align}
\text{(a) } \I_L(\rho) \geq 0, \quad \text{ and }\quad \text{(b) }\rho \text{ solves }\eqref{eq:forwardKolmogorovEquation}\Longleftrightarrow \I_L(\rho) = 0,
\label{eq:propertiesRateFunctional}
\end{align}
and consequently the equation ``$\I_L(\rho)=0$'' can be viewed as a variational characterisation of the forward Kolmogorov equation.

The Lagrangian~$\curlL$ is defined as the Legendre dual of a \emph{Hamiltonian} $\curlH:\curlP(\curlX)\times \Cb(\curlX)\rightarrow[0,\infty]$,  
\begin{align}
\curlL(\mu, s) := \sup_{\xi \in \Cb(\curlX)} \left\{\sum_{x \in \curlX} \xi(x) s(x) - \curlH(\mu, \xi)\right\}.
\label{eq:L=H^*}
\end{align}
In our setting of a Markov process on a discrete state space with generator $L$, the Hamiltonian is explicitly given by
\begin{align}
\curlH(\mu, \xi) := \sum_{x,y \in \curlX} \mu(x) L(x,y) \left[e^{\xi(y) - \xi(x)} - 1\right],
\label{eq:hamiltonian}
\end{align}
and by Legendre duality it has the alternative characterization
\begin{align}
\curlH(\mu, \xi) = \sup_{s \in \ell^1(\curlX)} \left\{\sum_{x \in \curlX} \xi(x) s(x) - \curlL(\mu, s)\right\}.
\label{eq:legendreTransform}
\end{align}

The following result places the preceding remarks in a rigorous context. We denote the space of right-continuous functions with left limits mapping $[0,T]$ into $\curlP(\curlX)$ by $D_{\curlP(\curlX)}[0,T]$, and the dual pairing between $\Cb(\curlX)$ and $\curlP(\curlX)$ by $\langle f,\mu\rangle = \sum_{x \in \curlX} f(x)\mu(x)$ for any $f\in \Cb(\curlX)$ and $\mu\in \curlP(\curlX)$, then the following result holds.

\begin{theorem}\label{thm:LDP}
	Let $\rho^N\in\curlP(\curlX)$ be the empirical process~\eqref{eq:empiricalProcesses} generated by $N\in\N$ independent Markov processes $(X^i)_{i = 1, \dots N}$ on the state space $\curlX$ with generator $L$.
	Furthermore, assume that the initial values $(\rho^N_0)_{N \in \N}$ are deterministic and converge in $\curlP(\curlX)$ to some $\rho_0$.
	Then, $(\rho^N)_{N \in \N}$ satisfies a large deviations principle in $D_{\curlP(\curlX)}[0,T]$ with rate functional $\I_L:\curlC([0,T];\curlP(\curlX))\rightarrow\R$ given by~\eqref{eq:rateFunctional}, and which has the alternative representation
	\begin{align}
	\I_L(\mu) = \sup_{f \in \curlC^1([0,T];\Cb(\curlX))} \left\{\langle f_T, \mu_T\rangle - \langle f_0, \mu_0\rangle - \int_0^T \Bigl(\langle\partial_t f_t, \mu_t\rangle + \curlH(\mu_t, f_t)\Bigr) \,dt\right\}
	\label{eq:rateFunctionalTheorem}
	\end{align}
	where $\mu \in \curlC([0,T];\curlP(\curlX))$ with $\mu|_{t = 0} = \rho_0$ and the Hamiltonian $\curlH$ is defined in~\eqref{eq:hamiltonian}.  Additionally, if for some $\mu \in \curlC([0,T];\curlP(\curlX))$ we have $\I_L(\mu) < \infty$, then  $t\mapsto\mu_t\in \curlP(\curlX)$ is absolutely continuous, and the rate functional can be reformulated as 
	\begin{align}\label{eq:BoundRF}
	\I_L(\mu) = \sup_{ f \in L^\infty(0,T;\Cb(\curlX))} \int_0^T \Bigl(\langle f_t, \partial_t \mu_t\rangle - \curlH(\mu_t, f_t) \Bigr)\,dt.
	\end{align}
\end{theorem}
The existence of the large-deviation principle is a reformulation of~\cite[Proposition~5.10]{kraaij18}, while the main statement of the theorem is the alternative characterization~\eqref{eq:rateFunctionalTheorem}; we give the proof in  Appendix~\ref{App-ProofLDP}. Appendix~\ref{app:Bochner} collects some results on absolutely-continuous curves and integration. 

\subsection{Main results}\label{sec:MainRes}
As mentioned earlier, the main goal of this work is to connect relative entropy, Fisher Information and large-deviation rate functional in the context of Markov processes on a discrete state space. While the connection between the relative entropy and the rate functional is fairly classical,  
it does not connect to the Fisher Information. As pointed out earlier, these objects have been connected recently in the case when $\curlX = \R^n$ and $L$ is a diffusion operator via the inequality (see~\cite[Chapter 2]{Sharma17} and~\cite[Section 2.5]{DLPSS-TMP} for details)
\begin{align}
\RelEnt(\mu_T \vert \rho_T) + \int_0^T \RF_L(\mu_s \vert \rho_s) \,ds \leq \RelEnt(\mu_0 \vert \rho_0) + \I_L(\mu),
\label{eq:FIR}
\end{align}
where $\mu$ is a measure-valued curve (such that the right-hand side of the estimate is well defined) and $\rho$ solves $\partial_t\rho=L^T\rho$. In~\cite{Sharma17} this relation is called the free-energy--relative-Fisher-Information--rate-functional (FIR) inequality, a terminology that we will use throughout this paper. 

We shall demonstrate in Section~\ref{subsec:failureClassicalFIR} that such an inequality already fails in fairly simple situations for a Markov jump process.  
To get around this issue, we generalise the notion of the relative Fisher Information. 

\begin{subequations}
	\label{eq:genRFI}
	\begin{defi}
		\label{def:genRFI-MAP}
		Let $\lambda\in(0,1)$.
		We define the \emph{generalised relative Fisher Information} $\RF^\lambda_L: \curlP(\curlX)\times \curlP(\curlX) \to [0,\infty]$ (corresponding to a generator $L$) as follows.
		\begin{enumerate}
			\item
			If $\rho, \mu\in\curlP_+(\curlX)$ and $\sup_{x\in \curlX} \max\{\mu(x)/\rho(x), \rho(x)/\mu(x)\}<\infty$, then 
			\begin{align}
			\label{eq:genRFIa}
			\RF^\lambda_L(\mu \vert \rho) &:= \sum_{x,y \in \curlX}
			L(x,y) \dfrac{ \mu(y)}{\rho(y)} \rho(x) - \dfrac{1}{\lambda} \curlH\left(\mu, \lambda \log\left(\dfrac{\mu }{\rho}\right)\right)\\
			&=\sum_{x,y \in \curlX}
			L(x,y) \left[\dfrac{ \mu(y)}{\rho(y)} \rho(x) - \mu(x)
			-\frac1\lambda \left( \mu(x)^{1-\lambda}\rho(x)^\lambda 
			\left(\dfrac{\mu(y)}{\rho(y)}\right)^\lambda -\mu(x) \right)\right].
			\label{eq:genRFIb}
			\end{align}
			Here $\curlH$ is the Hamiltonian~\eqref{eq:hamiltonian} that arises in the context of large deviations. 
			
			\item 
			If $\rho,\mu\in \curlP(\curlX)$, then 
			\begin{equation}
			\label{eq:genRFIc}
			\RF^\lambda_L(\mu \vert \rho) := \sum_{x,y \in \curlX}
			L(x,y) \psi_\lambda(x,y),
			\end{equation}
			where $\psi_\lambda$ is defined as 
			\[
			\psi_\lambda(x,y) := \begin{cases}
			\dfrac{ \mu(y)}{\rho(y)} \rho(x) - \mu(x)
			-\dfrac1\lambda \left( \mu(x)^{1-\lambda}\rho(x)^\lambda 
			\left(\dfrac{\mu(y)}{\rho(y)}\right)^\lambda -\mu(x) \right),
			&\text{if $\rho(y)>0$, and $\rho(x)>0$},\\
			+\infty & \text{\parbox{0.25\textwidth}{if $\rho(y)=0$, $\rho(x)>0$,\\ \hspace*{47pt} and $\mu(y)>0$,}}\\
			0 & \text{otherwise}.
			\end{cases}
			\]
		\end{enumerate}
	\end{defi}
\end{subequations}	
Both these definitions of the generalised relative Fisher Information are consistent, i.e.\ whenever both definitions apply, they give the same value (see Lemma~\ref{l:props-gRFI}). To motivate these definitions, we use the characterisation \eqref{eq:rateFunctionalTheorem} of the rate functional and reason formally as follows.  Let $\mu : [0,T] \rightarrow \curlP(\curlX)$ be a smooth curve with $\I_L(\mu)<\infty$ and $\rho : [0,T] \rightarrow \curlP(\curlX)$ be a smooth solution of the forward Kolmogorov equation~\eqref{eq:forwardKolmogorovEquation} such that $\log(\mu/\rho)$ is sufficiently regular. 
Using $f=\lambda\log(\mu/\rho)$ with $\lambda \in (0,1)$ in \eqref{eq:rateFunctionalTheorem}, we obtain 
\begin{align*}
\dfrac{1}{\lambda} \I_L(\mu) 
&\geq \sum_{x \in \curlX} \log\left(\dfrac{\mu_T(x)}{\rho_T(x)}\right) \mu_T(x) + \sum_{x \in \curlX} \log\left(\dfrac{\mu_0(x)}{\rho_0(x)}\right) \rho_0(x) \\
&\hspace*{5em}- \int_0^T \left(\sum_{x \in \curlX} \partial_t \log\left(\dfrac{\mu_t(x)}{\rho_t(x)}\right) \mu_t(x) + \frac{1}{\lambda}\curlH\left(\mu_t, \lambda\log\left(\dfrac{\mu_t}{\rho_t}\right) \right)\right)dt\\
&= \RelEnt(\mu_T \vert \rho_T) - \RelEnt(\mu_0 \vert \rho_0) +\int_0^T\left( 
\sum_{x,y \in \curlX} L(x,y) \mu_t(y)\dfrac{ \rho_t(x)}{\rho_t(y)} - \dfrac{1}{\lambda} \curlH\left(\mu_t, \lambda \log\left(\dfrac{\mu_t }{\rho_t}\right)\right) dt
\right),
\end{align*}
where the equality follows since
\begin{align*}
\sum_{x \in \curlX} \partial_t \log\left(\dfrac{\mu_t(x)}{\rho_t(x)}\right) \mu_t(x)=\sum_{x \in \curlX} \partial_t\mu_t(x) - \sum_{x \in \curlX} \dfrac{\mu_t(x)}{\rho_t(x)}(L^T\rho)(x) = 0-\sum_{x,y \in \curlX} L(x,y)\mu_t(y)\dfrac{\rho_t(x)}{\rho_t(y)}.
\end{align*}
The formal inequality above resembles~\eqref{eq:FIR}, where the integrand in the time integral is precisely the generalised Fisher Information given in \eqref{eq:genRFIa}. 
These formal calculations can and will be made rigorous, resulting in the first main result of this article which we now state. 
\begin{theorem}\label{thm:FIR}
	Let $\rho\in A\curlC([0,T];\curlP(\curlX))$ be a solution of \eqref{eq:forwardKolmogorovEquation} and $\mu \in \curlC([0,T];\curlP(\curlX))$ satisfy
	\begin{align*}
	\I_{L}(\mu)+\RelEnt(\mu_0\vert\rho_0)<\infty,
	\end{align*}
	with $\mu|_{t=0}=\mu_0$. 
	Then for any $\lambda \in (0,1)$ we have
	\begin{align}\label{eq:FIRwGeneralisedRFI}
	\RelEnt(\mu_T \vert \rho_T) + \int_0^T \RF_L^\lambda(\mu_t \vert \rho_t) \,dt \leq \RelEnt(\mu_0 \vert \rho_0) + \dfrac{1}{\lambda} \I_L(\mu).
	\tag{FIR$_\lambda$}
	\end{align}
\end{theorem}
It is important to note that the roles of $\mu$ and $\rho$ in the FIR inequality~\eqref{eq:FIRwGeneralisedRFI} cannot be interchanged, i.e.\  $\mu$ is a solution to the forward Kolmogorov equation and $\rho$ is arbitrary, since the relative entropy is not symmetric. 
As evident from the formal calculations above, the generalised relative Fisher Information~\eqref{eq:genRFI} is constructed such that  the proof of the FIR inequality goes through. 
In addition to satisfying~\eqref{eq:FIRwGeneralisedRFI}, the generalised Fisher Information has several favourable properties which we now summarise (see Section~\ref{subsec:motivationAndProofFIR}--\ref{subsec:propetiesGenRFI} for details).
\begin{theorem}\label{thm:RFproperties}
	For $\lambda\in(0,1)$, the generalised Fisher Information satisfies:		
	\begin{enumerate}[label=(\roman{*})]
		\item $\RF^\lambda_L$ is non-negative and lower-semicontinuous on $\curlP(\curlX)\times\curlP(\curlX)$.
		\item If $\mu,\rho\in\curlP(\curlX)$ with $\RF^\lambda_L(\mu \vert \rho)=0$, then $\mu$ is a constant multiple of $\rho$ on each connected component of the support of $\rho$. 
		In particular, if $\rho\in\curlP_+(\curlX)$, then $\mu=\rho$ on $\curlX$.
		\item $\RF_L^\lambda \to \RF_L$ as $\lambda\to0$ on $\curlP_+(\curlX) \times \curlP_+(\curlX)$ in the sense of Gamma convergence.
	\end{enumerate}
\end{theorem}
Whenever two measures $\rho$ and $\mu$ satisfy $\RF^\lambda_L(\mu \vert \rho)=0$, Theorem~\ref{thm:RFproperties}(ii) provides information on how they are related, similar to that of a logarithmic version of a Dirichlet form in continuous state spaces.
The name `generalised' Fisher Information is motivated by the fact that  we can recover the relative Fisher Information~\eqref{eq:RFIBasicDef} as a limit for $\lambda\rightarrow 0$ (cf.~Theorem~\ref{thm:RFproperties}(iii)). In addition to this asymptotic relation, the generalised and the classical relative Fisher Information can also be compared directly by an inequality in a fairly restrictive setting, thereby allowing us to prove a FIR inequality with the classical Fisher Information (see Section~\ref{sec:FIR-gFIR} for details). 

We point out that the FIR inequality bears similarity to the entropy-dissipation identity that arises in the context of reversible Markov processes and more generally gradient flows (see~\cite{mielkePeletierRenger14} for details). However in Theorem~\ref{thm:FIR} (and throughout this article) we do not assume the generator $L$ to be reversible and therefore our results go beyond the existing results on gradient flows. Additionally, the FIR inequality compares two curves, which is not the case for the entropy-dissipation identity.
\subsection{Application to coarse-graining  }\label{subsec:application} 
Coarse-graining is an umbrella term used for techniques which approximate a complex or high-dimensional system by a simpler or lower-dimensional one. While there are many formal techniques for achieving this (see~\cite{GivonKupfermanStuart04} and references therein), rigorous mathematical analysis is typically restricted to situations that exhibit explicit separation of temporal and/or spatial scales, i.e.\ the presence of fast and slow variables. In these situations, as the ratio of `fast' to `slow' increases, some form of averaging or homogenization allows one to remove the fast scales, and obtain a limiting system that focuses on the slow ones. Recently, a new variational technique based on studying the large-deviation rate functional has been introduced in~\cite{duongLamaczPeletierSharma17,Sharma17} to study coarse-graining limits arising in the context of diffusion processes (see Section~\ref{subsec:variationalFrameworkCG} for details). As mentioned earlier, in this paper we apply this variational technique to study a coarse-graining problem arising in the discrete setting (described below). The generalised Fisher Information~\eqref{eq:genRFI} and the FIR inequality~\eqref{eq:FIRwGeneralisedRFI} described in the last section play a crucial role in this study.

The coarse-graining problem we study here is inspired by kinetic Monte-Carlo methods in molecular dynamics (see~\cite[Chapter 5]{lahbabi13} for details). Consider a particle moving in a potential-energy landscape, which consists of small and large barriers as described in Figure~\ref{fig:energyLandscape}. The large energy barriers  introduce a natural scale-separation since it is harder for the particle to jump across them  compared to the smaller barriers. More precisely we can model the behaviour of such a particle as a Markov jump process on $\curlX = \curlY\times \curlZ$ where $\curlY$ corresponds to the states separated by the large energy barriers while $\curlZ$ is the part of the state space separated by small energy barriers. For simplicity, we assume that there is only one large barrier, i.e.\ $\curlY = \{0,1\}$ and finitely many small barriers corresponding to each of these large barriers, i.e.\ $\curlZ = \{1, \dots, n\}$. This intuitively means that the state space is divided up into two \emph{macro-states}, each of 
which contain $n\in\N$ easily accessible \emph{macro-states}.  

We consider the Markov process which evolves according to the generator  
\begin{align*}
\tilde L^\e= Q + \e C := \begin{pmatrix}Q_0 & 0 \\ 0 & Q_1\end{pmatrix} + \e \begin{pmatrix}D_0 & C_{0,1} \\
C_{1,0} & D_1 \end{pmatrix}, 
\end{align*}
where $Q$ and $C$ are $\e$-independent matrices with 
\begin{align}\label{CG-L-ass}
\forall x_1\in\curlX: \ \sum_{x_2\in\curlX}Q(x_1,x_2)=0=\sum_{x_2\in\curlX}C(x_1,x_2).
\end{align}
The diagonal matrix $D_y$, $y\in\curlY$, is constructed so that $C$ satisfies the aforementioned property, i.e.\ 
\begin{align*}
\forall z_1\in\curlZ: \ D_y(z_1,z_1):=-\sum_{z_2\in\curlZ} C_{y,1-y}(z_1,z_2).
\end{align*}
We assume that $Q_y$ is irreducible for every $y\in \curlY$ and $\tilde L^\e$ is irreducible. The irreducibility of $\tilde L^\e$ is equivalent to assuming that $C_{1,0}$ and $C_{0,1}$ have at least one positive entry.  
\begin{figure}
	\begin{center}
		\begin{tikzpicture}[scale=1.3]
		\draw [domain=-2.5:-0.5, samples=100] plot (\x, {cos(pi*\x r)*cos(pi*\x r)});
		\draw [domain=0.5:2.5, samples=100] plot (\x, {cos(pi*\x r)*cos(pi*\x r)});
		\draw [domain=-3:-2.5, samples=50] plot(\x, {3*cos(pi*\x r)*cos(pi*\x r)});
		\draw [domain=2.5:3, samples=50] plot(\x, {3*cos(pi*\x r)*cos(pi*\x r)});
		\draw [domain=-0.5:0.5, samples=50] plot(\x, {3*cos(pi*\x r)*cos(pi*\x r)});
		\draw (-3,-0.1) -- (-3,-0.3) (-3,-0.2)--(-1.5,-0.2) node[align=center,below=2pt]{macro-state}--(0,-0.2) (0,-0.1) -- (0,-0.3);
		\draw (1,-0.1) -- (1,-0.3) (1,-0.2)--(1.5,-0.2) node[align=center,below=2pt]{micro-state}--(2,-0.2) (2,-0.1) -- (2,-0.3);
		\end{tikzpicture}
		\caption{Energy landscape with two macro-states.}
		\label{fig:energyLandscape}
	\end{center}
\end{figure}
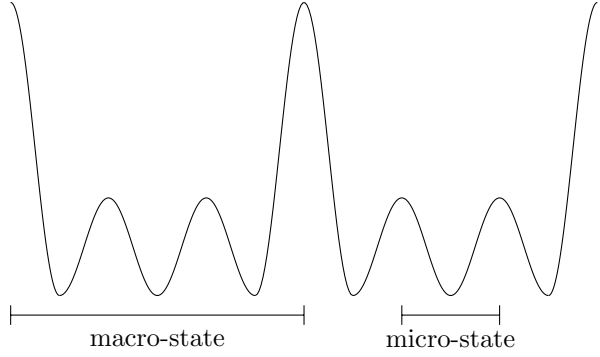

Now let us take a closer look at each of these components. The small parameter $\varepsilon > 0$ models the scale-separation arising due to the difference in the heights of the barriers. The matrix $Q_y\in\R^{n\times n}$ encodes the jumps  between macro-states within the $y$-th macro-state. The matrix
$C_{y,1-y}\in\R^{n\times n}$ encodes the transition from the $y$-th macro-state to $(1-y)$-th macro-state. The summability condition~\eqref{CG-L-ass} ensures that $\tilde L^\e$ is a generator, i.e.\ an operator satisfying \eqref{eq:conditionsGenerator1}.

When $\e$ is small, the dynamics of the particle evolving according to $\tilde L^\e$ splits into slow and fast components. The fast component moves the particle within a macro-state, and the slow component is visible as a rare jump to a different macro-state. Following~\cite{lahbabiLegoll13}, in order to focus on the slow component we rescale time by $\e^{-1}$ and arrive at 
\begin{align}
L^\e= \frac1\e Q +  C := \frac1\e\begin{pmatrix}Q_0 & 0 \\ 0 & Q_1\end{pmatrix} +  \begin{pmatrix}D_0 & C_{0,1} \\
C_{1,0} & D_1 \end{pmatrix}.	
\label{eq:fastReactionLimit}
\end{align}

The main goal of the second part of this work is to study the behaviour of the Markov jump process described by the forward Kolmogorov equation
\begin{align}
\begin{dcases}
\partial_t \mu^\varepsilon = (L^\varepsilon)^T \mu^\varepsilon, \\
\mu^\varepsilon_{t = 0} = \mu_0,
\end{dcases}
\label{eq:CGforwardKolmogorovEq}
\end{align}
in the limit  $\e\rightarrow 0$. In this limit it is natural to expect that the solution $\mu^\e$ equilibrates in each macro-state and the limit can be described by a jump process on $\curlY$, i.e.\ a two-point Markov jump process. In the second part of this article we make this intuition precise (see Section~\ref{sec:CGApplication} for details). 

To state the precise result we need to introduce two objects: (1) the stationary measure 
of~\eqref{eq:CGforwardKolmogorovEq}, denoted by $\pi^\e\in\curlP(\curlX)$, which exists since $L^\e$ is irreducible, and (2) the \emph{coarse-graining map} $\xi:\curlX\rightarrow\curlY$ as $\xi(x)=y$ for every $x=(y,z)\in\curlX$. 

For more details on this coarse-graining map see Section~\ref{sec:CGApplication}.
\begin{theorem}\label{thm:Intro-CG-Sol}
	Consider a sequence $\mu^\varepsilon\in \curlC([0,T];\curlP(\curlX))$ of solutions to~\eqref{eq:CGforwardKolmogorovEq}. Assume that the initial data satisfies
	\[
	\sup_{\e >0} \RelEnt(\mu_0^\e\vert\pi^\e) <\infty. 
	\]
	We then find for a subsequence (not relabeled)
	\begin{enumerate}
		\item(Compactness) The sequence  $\mu^\varepsilon \rightarrow \mu$ in $\curlM([0,T] \times \curlX)$, the space of non-negative, finite measures on $[0,T] \times \curlX$, with respect to the narrow topology, and $\xi_{\#}\mu^\varepsilon \rightarrow \xi_{\#}\mu$ in $\curlC([0,T];\curlP(\curlY))$ uniformly in time.
		\item (Local equilibrium) There exists $\hat\mu\in \curlC([0,T];\curlP(\curlY))$ such that for almost all $t\in [0,T]$ 
		\begin{align*}
		\forall y\in\curlY, \ A \subset \curlZ, \ \mu_t(\{y\}\times A) = \hat{\mu}_t(y) \pi_y(A),
		\end{align*}
		where for each $y\in\curlY$, $\pi_y \in \curlP(\curlZ)$ is the stationary measure corresponding to $Q_y$. Furthermore $\xi_{\#} \mu^\varepsilon \rightarrow \hat{\mu}$ in $\curlC([0,T];\curlP(\curlY))$ uniformly in time.
		\item (Limit dynamics) The limit $\hat\mu\in \curlC([0,T];\curlP(\curlY))$ solves 
		\[ \partial_t\hat\mu=L^T\hat\mu \]
		with the (limiting) generator 
		\begin{align*}
		L := \left(\begin{array}{cc}
		- \lambda_0 & \lambda_0 \\
		\lambda_1 & -\lambda_1
		\end{array}\right), \quad \lambda_y := \sum_{z,z' \in \curlZ} \pi_y(z) C_{y,1-y}(z,z').
		\end{align*}
	\end{enumerate} 
\end{theorem}

Note that we do not specify the topology on $\curlP(\curlX)$ in this result, since $\curlX$ is finite and thus $\curlP(\curlX)$ is a subset of a finite-dimensional space (also see Remark \ref{rem:topologies}).
Furthermore, we point out that this result is a special case of our analysis in Section~\ref{sec:CGApplication}, which also applies to the case of \emph{approximate solutions} (see Remark~\ref{rem:Sol-BoundRF} for details).
\subsection{Comparison with other work}We now comment on the novelties developed in this paper compared with other work.
\begin{enumerate}
	\item \textit{In comparison with other works on the FIR inequality.}
	As mentioned earlier, the idea of an FIR inequality connecting the free energy (which, in our case, is the relative entropy), the relative Fisher Information and the large deviation rate functional was discussed in the context of diffusion processes~\cite{bogachevRoecknerShaposhnikov16,duongLamaczPeletierSharma17,DLPSS-TMP,Sharma17}, although most of these works do not explicitly refer to this inequality as the FIR inequality.
	Our contribution lies in the extension of the FIR inequality to the discrete settings which is substantially different from  the diffusion case treated in the references above.
	The main difference is that the Hamiltonian in the discrete case has a different scaling behaviour which ensures that  the classical FIR inequality fails in the discrete setting (see Section~\ref{subsec:failureClassicalFIR} for details). For a more detailed review of these connections see Section~\ref{sec:ContStateSpace}.
	\item \textit{In comparison with other work on the example treated in this paper.} 
	The coarse-graining example introduced in Section \ref{subsec:application} is an averaging problem for Markov chains~\cite{PavliotisStuart08,lahbabiLegoll13}. In these references, martingale techniques are used to prove a pathwise convergence result while our proof relies on the variational framework given by the large deviations result. Although the convergence result in this work is weaker, we obtain an explicit local-equilibrium statement and 
	our result also applies to approximate solutions, i.e.\ curves with finite rate functional, rather than zero. This  allows us to work with a larger class of measures (see Remark~\ref{rem:Sol-BoundRF}). This latter property also distinguishes our approach from other classical strategies  such as geometric singular perturbation theory, see for instance~\cite{kuehn15}. 
	\item \textit{Comparison with variational evolutionary methods.}
	In recent years,  variational-evolutionary structures akin to gradient flows have been developed for forward Kolmogorov equations on finite state-spaces~\cite{maas11,mielke11,mielke13,chowHuangLiZhou}. This structure can also be used to investigate singular limits ~\cite{sandierSerfaty04,serfaty11,mielke16}.
	However these structures are limited to reversible Markov chains, while the approach discussed in this paper does not require reversibility since we only use the variational structure provided by the large-deviations principle.
	\item \textit{Quantitative coarse-graining.} 
	As in the diffusion case~\cite{DLPSS-TMP,Sharma17}, a natural next step is to derive explicit error estimates for `finite' scale separation. However, the strategy to obtain those estimates does not use the full FIR inequality but only a related result inspired by~\cite{yau91} and is thus omitted in this paper. For details we refer to~\cite[Chapter 8]{hilder17}. 
\end{enumerate}

\subsection{Outline of the article}
In the rest of the paper we present the details of the ideas introduced above. 
In Section~\ref{sec:genRFIandFIR} we construct the generalised Fisher Information and prove the FIR inequality. In Section~\ref{sec:CGApplication} we study the coarse-graining problem using the variational technique developed in~\cite{duongLamaczPeletierSharma17}.
Section~\ref{sec:conclusionDiscussion} provides further discussions and generalisations and certain details on the rate functional are discussed in Appendix~\ref{App-ProofLDP}. In Appendix~\ref{app:Bochner} we collect some results on integration in infinite-dimensional spaces and in Appendix \ref{app:positivity} we provide a result on positivity of solutions for irreducible generators.

\section{Generalised relative Fisher Information and FIR inequality}\label{sec:genRFIandFIR}
In Section~\ref{subsec:failureClassicalFIR} we discuss a simple example where the FIR inequality fails when working with the classical relative Fisher Information~\eqref{eq:RFIBasicDef}, following which we prove the FIR inequality with the generalised relative Fisher Information~\eqref{eq:genRFI} in Section~\ref{subsec:motivationAndProofFIR}. We then prove the main properties of the the generalised Fisher Information in Section~\ref{subsec:propetiesGenRFI}. Finally in Section~\ref{sec:ContStateSpace} we connect these ideas to diffusions and compare to existing results in the literature. 

\begin{rem}[Extension to finite measures]\label{rem:finiteMeasures}
	We restrict the treatment in what follows to probability measures to keep the notation simple.		
	However, the definition as well as the properties of the generalised Fisher Information can be generalised to non-negative, finite measures with no additional difficulties.
\end{rem}

\subsection{Failure of FIR inequality with relative Fisher Information}\label{subsec:failureClassicalFIR}

Before we present the proof of the FIR inequality with the generalised Fisher Information (described in Theorem~\ref{thm:FIR}), we first show a simple example where such an inequality~\eqref{eq:FIR} fails when working with the `classical' relative Fisher Information~\eqref{eq:RFIBasicDef}.  Note that this is distinctly different from the case of diffusions on continuous state space where the FIR inequality holds for the relative Fisher Information (for a detailed discussion see Section~\ref{sec:ContStateSpace}).

The idea is to construct a sequence of curves for which the rate functional stays bounded while the classical relative Fisher Information is unbounded in the limit, which would prove that the FIR inequality does not hold in this setting. We consider a two-point space $\curlX=\{0,1\}$ and a generator given by
\begin{align*}
L=\begin{pmatrix}-a & a \\ b & -b\end{pmatrix},
\end{align*}
for $a,b>0$. Furthermore we consider a constant-in-time curve $\mu\in\curlP(\curlX)$. For any $f\in \Cb(\curlX)$ and $s:=f(0)-f(1)$, the Hamiltonian~\eqref{eq:hamiltonian} can be written as
\begin{align*}
\curlH(\mu, f) = a\mu(0) \left(e^{-s} - 1\right) + b(1-\mu(0)) \left(e^s -1\right).
\end{align*}
There exists a constant $c>0$ such that for any $f$ and $\mu$ we have $\curlH(\mu, f)>-c$. Therefore using the definition of the rate functional~\eqref{eq:BoundRF} we find
\begin{align*}
\forall \mu\in\curlP(\curlX): \ \I_L(\mu) = \sup_{f \in L^\infty([0,T];\Cb(\curlX))} \int_0^T - \curlH(\mu, f_t) \,dt \leq  cT.
\end{align*}
Next let us look at the classical relative Fisher Information~\eqref{eq:RFIBasicDef} with $\rho=(a+b)^{-1}(b,a) \in\curlP_+(\curlX)$ which satisfies $L^T\rho=0$. Writing $\mu = (\mu_0,1-\mu_0)$ and $\rho =(\rho_0,1-\rho_0)$ we find
\begin{align*}
\RF_L(\mu \vert \rho) &= a\left[\dfrac{(1-\mu_0) \rho_0}{1-\rho_0} - \mu_0 - \mu_0 \log\left(\dfrac{(1-\mu_0) \rho_0}{\mu_0 (1-\rho_0)}\right)\right]\\
&\hspace*{5em}+ b \left[\dfrac{\mu_0 (1-\rho_0)}{\rho_0} - (1-\mu_0) - (1-\mu_0) \log\left(\dfrac{\mu_0(1-\rho_0)}{(1-\mu_0)\rho_0}\right)\right].
\end{align*}
Choosing a sequence $(\mu^n)$ with $\mu_0^n\rightarrow 0$,  we have $\RF_L(\mu^n \vert \rho)\to \infty$, and therefore for any $C>0$, $\RF_L(\mu^n\vert\rho) \geq C \I(\mu)$ for a large enough $n$. As a result, the FIR inequality with the classical relative Fisher Information~\eqref{eq:FIR} does not hold in the discrete setting in general.  

\begin{rem}
	Note that this example did not exploit any pathological behaviour of the generator and works for all irreducible  generators $L$ on this two-point state space. Therefore we do not expect that there is a simple restriction on the class of admissible generators such that the FIR inequality~\eqref{eq:FIR} holds. A careful look at the example reveals that the FIR inequality fails since $\log(\mu_0/\rho_0) \rightarrow -\infty$ as $\mu_0 \rightarrow 0$, and if such choices of~$\mu$ are excluded than an FIR inequality with the relative Fisher Information might hold. This is indeed the case, as will be discussed in Lemma~\ref{lem:FIRrestrictedMeasure}. 
	
	On the other hand, the generalised relative Fisher Information~\eqref{eq:genRFI} does not suffer from the issue above since in this setting for any fixed $\lambda\in (0,1)$ we find
	\begin{align}
	\RF_L^\lambda(\mu \vert \rho) =  0 - \dfrac{1}{\lambda} \curlH\left(\mu, \lambda \log\left(\dfrac{\mu }{\rho}\right)\right) < \dfrac{c}{\lambda}.
	\end{align}
	It is not a coincidence that the FIR inequality holds for the generalised Fisher Information, as we prove below. 
\end{rem}

\subsection{FIR inequality with generalised relative Fisher Information}\label{subsec:motivationAndProofFIR}

In what follows we first prove an auxiliary lemma on the structure of the generalised relative Fisher Information, which we use in Lemma~\ref{l:props-gRFI} to study the consistency of its definition and discuss some simple properties. We conclude this section by giving the proof of Theorem~\ref{thm:FIR}. 

For any $\rho(y),\rho(x)>0$, the function $\psi_\lambda$ in \eqref{eq:genRFIc} may be rewritten as
\begin{align}\label{eq:psi_r}
\psi_\lambda(x,y) = \frac{r_\lambda(v(x),v(y))}{\lambda}\rho(x),\qquad v=\frac{\mu}{\rho},
\end{align}
where $(\xi,\eta)\mapsto r_\lambda(\xi,\eta):= (1-\lambda)\xi-\xi^{1-\lambda}\eta^\lambda + \lambda\eta$. 

\begin{lem}\label{lem:r_lambda}
	For any $\lambda\in(0,1)$, the function $r_\lambda :[0,\infty)\times [0,\infty)\to \R$ defined by 
	\[
	r_\lambda(\xi,\eta)=(1-\lambda)\xi-\xi^{1-\lambda}\eta^\lambda + \lambda\eta,
	\]
	satisfies the following properties:
	\begin{enumerate}[label=(\roman*)]
		\item $r_\lambda\ge 0$ on $[0,\infty)\times [0,\infty)$;
		\item $r_\lambda(\xi,\eta)=0$ if and only if $\xi=\eta$;
		\item For any $\xi,\eta\ge 0$, the function $\lambda\mapsto \lambda^{-1}r_\lambda(\xi,\eta)$ is monotonically decreasing on $(0,1)$;
		\item For any $\xi,\eta>0$, $\lim_{\lambda \rightarrow 0} \lambda^{-1}r_\lambda(\xi,\eta)=\eta-\xi+\xi\log(\frac{\xi}{\eta})$ monotonically increasing.
	\end{enumerate}
\end{lem}
\begin{proof}
	\begin{enumerate}[label=$(\roman*)$]
		\item For any $\lambda\in(0,1)$ and $\xi,\eta\ge 0$, the Young's inequality yields 
		\[
		\xi^{1-\lambda}\eta^\lambda \le (1-\lambda)\xi + \lambda\eta,
		\]
		and the non-negativity of $r_\lambda$ follows by simply rearranging the terms.
		\item The reverse implication follows trivially by inserting $\xi=\eta$.
		Now assume that $r_\lambda(\xi,\eta)=0$. If $\xi=0$, it follows that $\eta=0$ and vice versa. Therefore without the loss of generality we assume that $\xi>0$, which implies that $\eta>0$. By rewriting
		\[
		r_\lambda(\xi,\eta)=\xi\bigl((1-\lambda)-s^\lambda + \lambda s\bigr),\qquad s=\eta/\xi,
		\]
		and noting that the function $s\mapsto s^\lambda $ is strictly concave on $(0,\infty)$, we deduce that the expression within the bracket vanishes if and only if $s=1$, i.e.\ $\eta=\xi$.
		\item If $\xi=0=\eta$, there is nothing to show. Suppose $\xi=0$, then $\lambda^{-1}r_\lambda(\xi,\eta)= \eta$, i.e.\ $\lambda^{-1}r_\lambda(\xi,\eta)$ is constant in $\lambda$ and therefore monotonically decreasing. If $\eta=0$ and $\xi>0$, then $\lambda^{-1}r_\lambda(\xi,\eta)= (1/\lambda-1)\xi$, which is monotonically decreasing in $\lambda$ since $\lambda\mapsto 1/\lambda$ is monotonically decreasing.
		For $\xi,\eta>0$, we begin by observing that $\lambda\mapsto \lambda^{-1}r_\lambda(\xi,\eta)\in \curlC^1((0,1))$, with 
		\[
		\frac{d}{d\lambda} \frac{r_\lambda(\xi,\eta)}{\lambda} = \frac{\xi}{\lambda^2}\bigl( s^\lambda - 1 - s^\lambda\log s^\lambda \bigr),\qquad s=\eta/\xi.
		\]
		Since $\alpha\mapsto \alpha \log \alpha$ is convex on $(0,\infty)$, it follows that  $s^\lambda\log s^\lambda \ge s^\lambda - 1$, and therefore  $\lambda^{-1}r_\lambda(\xi,\eta)$ is monotonically decreasing in $\lambda$.
		\item Let $\xi,\eta>0$ and set $s=\eta/\xi$. Using l'Hospital's formula it follows that
		\begin{align}\label{eq:r_lam-lim}
		\lim_{\lambda \rightarrow 0}\frac{r_\lambda(\xi,\eta)}{\lambda} = \eta-\xi - \xi\lim_{\lambda \rightarrow 0}\left(\frac{s^\lambda-1}{\lambda}\right)  = \eta-\xi - \xi\lim_{\lambda \rightarrow 0}\left(\frac{e^{\lambda\log(s)}-1}{\lambda}\right) = \eta-\xi - \xi \log(s),
		\end{align}
		The monotonically increasing convergence holds due to $(iii)$. 
	\end{enumerate}
\end{proof}

\begin{lem}
	\label{l:props-gRFI}
	The two definitions in Definition~\ref{def:genRFI-MAP} are consistent; that is, whenever both definitions apply, they give the same value. Additionally, 
	\begin{enumerate}[label=(\roman{*})]
		\item $\RF^\lambda_L(\mu \vert \rho) \geq 0$ for all $\mu,\rho\in\curlP(\curlX)$;
		\item $\RF^\lambda_L$ is lower-semicontinuous on $\curlP(\curlX)\times\curlP(\curlX)$.
	\end{enumerate}
\end{lem}
\begin{proof}		
	Using the Hamiltonian~\eqref{eq:hamiltonian}, it is easy to check that the definitions~\eqref{eq:genRFIa} and~\eqref{eq:genRFIb} agree for any $\rho, \mu \in \curlP_+(\curlX)$ with $\sup_{x \in \curlX} \max\{\mu(x) / \rho(x), \rho(x) / \mu(x)\} < \infty$, which proves the consistency of Definition~\ref{def:genRFI-MAP}. 
	\begin{enumerate}[label=$(\roman{*})$]
		\item Since $\psi_\lambda(x,x) = 0$ for all $x$, the diagonal in the double sum in~\eqref{eq:genRFIc} vanishes. So we consider $x\not=y$, for which $L(x,y)\geq0$.
		If $\mu(x) =0$ or $\rho(y)=0$, then $\psi_\lambda(x,y)\geq0$; if $\rho(y)>0$, 
		$\psi_\lambda(x,y)=0$ if $\rho(x)=0$ and $\psi_\lambda(x,y)\ge0$ (due to \eqref{eq:psi_r} and the non-negativity of $r_\lambda$ in Lemma~\ref{lem:r_lambda}) if $\rho(x)>0$. Therefore $L(x,y)\psi_\lambda(x,y)\geq0$ for all $x,y$, and $\RF^\lambda_L(\mu\vert\rho)\geq0$. 
		\item Let $((\mu^n,\rho^n))_{n\in\N}\subset \curlP(\curlX)\times \curlP(\curlX)$ be a sequence that converges to $(\mu,\rho)$. In particular, $\mu^n(x)\to\mu(x)$ and $\rho^n(x)\to\rho(x)$ for every $x\in\curlX$ (cf.~Remark~\ref{rem:weak-strong-equivalence}).
		
		Now let $x\in\curlX$ be arbitrary and consider $y\in\curlX$ with $L(x,y)>0$. For simplicity, we denote
		\[
		\psi^n(x,y) = \dfrac{ \mu^n(y)}{\rho^n(y)} \rho^n(x) - \mu^n(x)
		-\dfrac1\lambda \left( \mu^n(x)^{1-\lambda}\rho^n(x)^\lambda 
		\left(\dfrac{\mu^n(y)}{\rho^n(y)}\right)^\lambda -\mu^n(x) \right).
		\]
		{\bf Case 1:} ($\rho(y)=\alpha>0$) Due to the pointwise convergence, there exists an $\alpha'>0$ such that $\rho^n(y)>\alpha'$ for sufficiently large $n$. In this case, we easily conclude that $ \psi^n(x,y)\to \psi(x,y)$ as $n\to\infty$.
		
		{\bf Case 2:} ($\rho(y)=0$, $\rho(x),\mu(y)\ge \beta>0$) As before,  there exists a $\beta'>0$ such that $\rho^n(x),\mu^n(y)>\beta'$ for sufficiently large $n$. Further, we have $\mu(x),\rho(x)\in[0,M]$ for all $x\in\curlX$, with some $M \geq 1$. Therefore,
		\[
		\psi^n(x,y) \ge (\beta')^2\frac{1}{\rho^n(y)} - M - \dfrac{1}{\lambda}M^{1+\lambda} \left(\dfrac{1}{\rho^n(y)}\right)^\lambda = \frac{1}{\rho^n(y)} \underbrace{\biggl[ (\beta')^2 - \dfrac1\lambda M^{1+\lambda} (\rho^n(y))^{1-\lambda}\biggr]}_{(*)} - M.
		\]
		Since $(\rho^n(y))^{1-\lambda}\to 0$ as $n\to\infty$, it follows that $(\beta')^2\ge (*)\ge \delta$ for some $\delta>0$ and sufficiently large $n$. Consequently, $\psi^n(x,y)\to\infty$ as $n\to\infty$.
		
		The other cases are trivial since $\psi^n(x,y)\ge 0$. An application of Fatou's lemma yields
		\[
		\liminf_{n \to \infty}\RF_L^\lambda(\mu^n \vert \rho^n) \ge \sum_{x,y \in \curlX} L(x,y) \liminf_{n \to \infty} \psi^n(x,y) \ge \sum_{x,y \in \curlX} L(x,y) \psi(x,y) = \RF_L^\lambda(\mu \vert \rho),
		\]
		thereby concluding the proof.
	\end{enumerate}
\end{proof}
We are now in a position to prove the first main result of this paper.

\begin{proof}[Proof of Theorem~\ref{thm:FIR}]
	We proceed by approximation. Let $\rho\in A\curlC([0,T];\curlP(\curlX))$ be a solution of~\eqref{eq:forwardKolmogorovEquation}. Since we assume the generator $L$ to be bounded~\eqref{eq:conditionsGenerator3} and irreducible~\eqref{eq:conditionsGenerator2}, it follows that  $\rho_t(x)>0$ for any $t>0$ and $x\in\curlX$ (see Lemma~\ref{lem:Pos-fKol} for a proof). 
	Without loss of generality we can assume that $\mu\in A\curlC([0,T];\curlP(\curlX))$, since by Theorem~\ref{thm:LDP} this is implied by $\I_L(\mu)<\infty$. Using Lemma~\ref{l:Bochner:1} we find $\rho,\mu\in W^{1,1}(0,T;\ell^1(\curlX))$ and therefore $\partial_t\rho,\partial_t\mu\in L^1(0,T;\ell^1(\curlX))$.
	
	For $\e>0$ and $\delta>0$, define the function $\rho^\e_t(x) := \rho_t(x)+\e\mu_t(x)$. Since $\mu\ll \rho^\e$, we can define the density 
	\[
	v_t^\e(x) := \frac{\mu_t(x)}{\rho_t^\e(x)} \in \left[0,\frac1\e\right].
	\]
	Note that $v^\e_t(x)\to \mu_t(x)/\rho_t(x)$ as $\e\to0$ for all $x\in\curlX$ and $t > 0$.

	Since $\log(v^\varepsilon + \delta)\in L^\infty(0,T;\Cb(\curlX))$ for any $\delta\in(0,1)$, using the representation \eqref{eq:BoundRF} we find 
	\[
	\frac{1}{\lambda }\I_L(\mu) \ge \int_0^T \langle \log(v_t^\varepsilon+ \delta),\partial_t\mu_t\rangle - \frac{1}{\lambda}\curlH(\mu_t, \lambda\log(v_t^\varepsilon+ \delta)) \,dt.
	\]
	We split the proof into two steps, where the first step deals with passing $\delta\to 0$ and the second step with passing  $\varepsilon\to 0$.
	
	{\bf Step 1:} Taking the liminf ($\delta\to 0$) in the previous inequality yields
	\begin{align*}
	\frac{1}{\lambda }\I_L(\mu) &\ge 
	\liminf_{\delta \rightarrow 0} \biggl\{ \int_0^T \langle \log(v_t^\varepsilon+ \delta),\partial_t\mu_t\rangle\,dt \biggr\} - \frac{1}{\lambda}\limsup_{\delta \rightarrow 0} \biggl\{ \int_0^T \curlH(\mu_t, \lambda\log(v_t^\varepsilon+ \delta)) \,dt \biggr\} = (I) - \frac{1}{\lambda}(II).
	\end{align*}
	We now study both these terms.
	
	{\it Part $(I)$:} Define the function $g_{\e,\delta}:[0,\infty)\times (0,\infty)\to\R$ by
	\[
	g_{\e,\delta}(\eta,\xi) := \eta \log\left( \frac\eta{\e\eta+\xi} +\delta\right).
	\]
	For fixed $\e,\delta$, the function $g_{\e,\delta}$ is globally Lipschitz on $A:= [0,\infty)\times (0,\infty)$, and differentiable at each $(\eta,\xi)\in A$.
	Since $\rho_t(x)>0$ for all $t>0$ and $x\in \curlX$, by Lemma~\ref{l:chain-rule-ell1} the function $t\mapsto g_{\e,\delta}(\mu_t,\rho_t) = \mu_t(x) \log(v_t^\varepsilon(x) + \delta)$ is an element of $A\curlC([0,T]; \ell^1(\curlX))$, and the following chain rule holds for almost every $t\in[0,T]$:
	\[
	\frac{d}{dt} \sum_{x\in\curlX} \mu_t(x)\log(v_t^\varepsilon(x) + \delta) = \sum_{x \in \curlX}\left(\frac{v_t^\varepsilon(x)}{v_t^\varepsilon(x) + \delta} + \log(v_t^\varepsilon(x) + \delta)\right)\partial_t\mu_t(x) 
	- \sum_{x\in \curlX} v_t^\varepsilon(x)\frac{v_t^\varepsilon(x)}{v_t^\varepsilon(x) + \delta}\partial_t\rho^\e_t(x).
	\]
	From this chain rule we easily deduce 
	\begin{equation}\label{eq:FIR-Aux1}
	\begin{aligned}
	\int_0^T \langle\log(v_t^\varepsilon + \delta),\partial_t\mu_t\rangle\,dt &= \sum_{x\in\curlX} \mu_T(x)\log(v_T^\varepsilon(x) + \delta) - \sum_{x\in\curlX} \mu_0(x)\log(v_0^\varepsilon(x) + \delta) \\
	&  \ \ - \int_0^T \sum_{x \in \curlX}\frac{v_t^\varepsilon(x)}{v_t^\varepsilon(x) + \delta}\partial_t\mu_t(x)\,dt + \int_0^T\sum_{x \in \curlX} v_t^\varepsilon(x)\frac{v_t^\varepsilon(x)}{v_t^\varepsilon(x) + \delta}\partial_t\rho^\e_t(x)\,dt.
	\end{aligned}
	\end{equation}
	We now pass to the limit $\delta\to0$ in each of the terms on the right-hand side.
	
	Since $\partial_t\mu,\partial_t\rho^\e\in L^1(0,T;\ell^1(\curlX))$ and $v^\varepsilon\in L^\infty(0,T;\Cb(\curlX))$ we may pass to the limit $\delta\to 0$ using the dominated convergence theorem to obtain
	\begin{gather*}
	\int_0^T\sum_{x \in \curlX} v_t^\varepsilon(x)\frac{v_t^\varepsilon(x)}{v_t^\varepsilon(x) + \delta}\partial_t\rho^\e_t(x)\,dt \quad\xrightarrow{\delta\to 0}\quad  \int_0^T\sum_{x \in \curlX} v_t^\varepsilon(x)\partial_t\rho^\e_t(x)\,dt =  \int_0^T\sum_{x \in \curlX} v_t^\varepsilon(x)\left[ (L^T\rho_t)(x) + \e\partial_t\mu_t\right] \,dt.
	\end{gather*}
	A similar argument gives 
	\begin{gather*}
	\int_0^T \sum_{x \in \curlX}\frac{v_t^\varepsilon(x)}{v_t^\varepsilon(x) + \delta}\partial_t\mu_t(x)\,dt \quad\xrightarrow{\delta\to 0}\quad 
	\int_0^T \sum_{x \in \curlX}\Indicator\{\mu_t(x)>0\}\, \partial_t\mu_t(x)\,dt.
	\end{gather*}
	This limit is equal to zero, as we now show using another application of the dominated convergence theorem shows. Let $H_m:\R\to[0,1]$ be a smooth approximation of the Heaviside function $H$ with $H_m(s) = 0$ for $s\leq 0$ and $H_m(s)\uparrow 1$ for $s>0$ as $m\to\infty$; set $f_m(s) = \int_0^s H_m(\sigma)\, d\sigma$. Since $f_m$ is Lipschitz, $t\mapsto f_m(\mu_t(\cdot))$ is again absolutely continuous by Lemma~\ref{l:chain-rule-ell1}, and we have the chain rule
	\[
	\sum_{x \in \curlX}\Bigl[f_m(\mu_T(x)) - f_m(\mu_0(x))\Bigr]
	= \int_0^T \sum_{x\in \curlX} H_m(\mu_t(x))\partial_t \mu_t(x)\, dt.
	\]
	Using the dominated convergence theorem on both sides, we pass to the limit $m\to \infty$ to find
	\[
	0 = \sum_{x \in \curlX}\Bigl[\mu_T(x) - \mu_0(x)\Bigr]
	= \int_0^T \sum_{x\in \curlX} \Indicator\{\mu_t(x)>0\}\, \partial_t \mu_t(x)\, dt.
	\]

	Turning to  the first term in~\eqref{eq:FIR-Aux1}, using $\mu_t(x)\log(v_t^\varepsilon(x) + \delta) \ge \mu_t(x)\log(v_t^\varepsilon(x))$ for any $(t,x)\in[0,T]\times \curlX$, we find
	\[
	\sum_{x\in\curlX} \mu_t(x)\log(v_t^\varepsilon(x) + \delta) \ge \sum_{x\in\curlX} \mu_t(x)\log(v_t^\varepsilon(x)) = \RelEnt(\mu_t \vert \rho_t^\varepsilon).
	\]
	At time zero, the finiteness of $\RelEnt(\mu_0\vert\rho_0)$ implies that whenever $\mu_0(x)>0$ we have $\rho_0(x)>0$, and therefore the density $v_0(x) := \mu_0(x)/\rho_0(x)$ is well-defined $\mu_0$-almost-everywhere.
	Using the concaveness and monotonicity of the natural logarithm, for the second term in~\eqref{eq:FIR-Aux1} we find
	\begin{align*}
	\sum_{x\in\curlX} \mu_0(x)\log(v_0^\varepsilon(x) + \delta) &= \sum_{x \in \curlX} v_0^\varepsilon(x)\log(v_0^\varepsilon(x) + \delta)\rho_0^\varepsilon(x)
	\le \sum_{x \in \curlX} \mu_0(x)\log(v_0^\varepsilon(x)) + \delta(1+\varepsilon) \\
	&\le \sum_{x \in \curlX} \mu_0(x)\log(v_0(x)) + \delta(1+\varepsilon) = \RelEnt(\mu_0 \vert \rho_0) + \delta(1+\varepsilon),
	\end{align*}
	where we have used $\rho_0^\varepsilon\ge \rho_0$ to arrive at the second inequality. 
	Altogether, we obtain
	\begin{align*}
	\liminf_{\delta \rightarrow 0} \int_0^T \langle\log(v_t^\varepsilon + \delta),\partial_t\mu_t\rangle\,dt &\ge \RelEnt(\mu_T \vert \rho_T^\varepsilon) - \RelEnt(\mu_0 \vert \rho_0) +  \int_0^T\sum_{x \in \curlX} v_t^\varepsilon(x)\left[ (L^T\rho_t)(x) + \e \partial_t\mu_t(x)\right]\,dt,
	\end{align*}
	which concludes part $(I)$.
	
	{\it Part $(II)$:} Using the definition~\eqref{eq:hamiltonian} of the Hamiltonian, and $\sum_{y\in\curlX}L(x,y)=0$ we find
	\begin{align*}
	\curlH(\mu_t, \lambda\log(v_t^\varepsilon + \delta)) &= \sum_{x,y \in \curlX} \mu_t(x) L(x,y) \left[e^{\lambda\log(v_t^\varepsilon + \delta)(y) - \lambda\log(v_t^\varepsilon + \delta)(x)} - 1\right] \\
	&= \sum_{x,y \in \curlX} \mu_t(x) L(x,y) \biggl(\frac{v_t^\varepsilon(y) + \delta}{v_t^\varepsilon(x) + \delta}\biggr)^\lambda   \\
	&=\sum_{x,y \in \curlX} \rho_t^\varepsilon(x)(v_t^\varepsilon(x))^{1-\lambda} L(x,y) \biggl(\frac{v_t^\varepsilon(x)}{v_t^\varepsilon(x) + \delta}\biggr)^\lambda(v_t^\varepsilon(y) + \delta)^\lambda.
	\end{align*}
	We have the upper bound
	\begin{align*}
	\left| \rho_t^\varepsilon(x)(v_t^\varepsilon(x))^{1-\lambda} L(x,y) \biggl(\frac{v_t^\varepsilon(x)}{v_t^\varepsilon(x) + \delta}\biggr)^\lambda(v_t^\varepsilon(y) + \delta)^\lambda\right| \leq \varepsilon^{\lambda-1} (\varepsilon^{-1}+1)^\lambda\, \rho^\e(x) |L(x,y)|,
	\end{align*}
	where we have used $|v^\e_t|\leq \varepsilon^{-1}$ and $\delta\in (0,1)$. Note that the right-hand side is an element of $\ell^1(\curlX\times\curlX)$ since $\rho^\e\in \ell^1(\curlX)$ and $L$ satisfies~\eqref{eq:conditionsGenerator2}.
	Using the dominated convergence theorem we find
	\begin{align*}
	\limsup_{\delta \rightarrow 0} &\int_0^T \curlH(\mu_t, \lambda\log(v_t^\varepsilon + \delta))\,dt  = 
	\int_0^T \sum_{x,y \in \curlX} \rho_t^\varepsilon(x)(v_t^\varepsilon(x))^{1-\lambda} L(x,y)(v_t^\varepsilon(y))^\lambda.
	\end{align*}
	This concludes part $(II)$.
	
	Putting both the parts together, we obtain
	\begin{align}
	\frac{1}{\lambda }\I_L(\mu) &\ge (I) - \frac{1}{\lambda}(II) \notag \\
	&\ge \RelEnt(\mu_T \vert \rho_T^\e) - \RelEnt(\mu_0 \vert \rho_0) \notag\\
	&\hspace*{5em}+ \int_0^T \sum_{x,y \in \curlX} L(x,y)
	\rho_t^\e(x)\left[  v_t^\e(y) - \frac1\lambda (v_t^\e(x))^{1-\lambda}(v_t^\e(y))^\lambda\right] dt + \e\int_0^T \sum\limits_{x\in\curlX}v_t^\varepsilon(x)\partial_t\mu_t(x) dt \notag \\
	&= \RelEnt(\mu_T \vert \rho_T^\e) - \RelEnt(\mu_0 \vert \rho_0) 
	+ \int_0^T \RF_L^\lambda(\mu_t \vert \rho_t^\e)\, dt + \e\int_0^T \sum\limits_{x\in\curlX}v_t^\varepsilon(x)\partial_t\mu_t(x) dt,
	\label{ineq:FIR-intermediate}
	\end{align}
	where in the final identity we used the property $\sum_{y\in \curlX} L(x,y)=0$ and \eqref{eq:psi_r}. This inequality clearly resembles the FIR inequality.

	{\bf Step 2:} We now take the limit $\e\to0$. For any $t\in(0,T]$ we have
	\begin{align*}
	\RelEnt(\mu_t \vert \rho_t^\e) &= \sum_{x\in\curlX} \mu_t(x)\log(v_t^\e(x)) = \sum_{x\in\curlX} v_t^\e(x)\log(v_t^\e(x))\rho_t^\e(x) \\		
	&= \sum_{x\in\curlX} \big[v_t^\e(x)(\log(v_t^\e(x))-1) + 1\big]\rho_t^\e(x) +\sum_{x\in\curlX} [\mu_t(x) - \rho_t^\e(x)]\\
	&= \sum_{x\in\curlX} \big[v_t^\e(x)(\log(v_t^\e(x))-1) + 1\big]\rho_t^\e(x) - \e.
	\end{align*}
	The final inequality follows since $\sum_{x\in \curlX} \rho^\e_t(x)  = 1+\e$. The summand in the final  right-hand side is non-negative, and for each $x$ and $t$ such that $\rho_t(x)>0$ we have $v_t^\e(x)\to v_t(x) = \mu_t(x)/\rho_t(x)$ for $\e\to 0$. We therefore apply Fatou's lemma to obtain
	\begin{align*}
	\liminf_{\e \rightarrow 0} \RelEnt(\mu_t \vert \rho_t^\e) 
	&\geq \liminf_{\e \rightarrow 0} \sum_{x\in\curlX} \big[v_t^\e(x)\log(v_t^\e(x))- v_t^\e(x) + 1\big]\rho_t^\e(x) \\
	&\geq \liminf_{\e \rightarrow 0} \sum_{x\in\curlX} \big[v_t^\e(x)\log(v_t^\e(x))- v_t^\e(x) + 1\big]\rho_t(x) \\
	&= \sum_{x\in\curlX} \big[v_t(x)(\log(v_t(x))-1) + 1\big]\rho_t(x) = \RelEnt(\mu_t \vert \rho_t).
	\end{align*}
	As for the other expression, we use the non-negativity and lower-semicontinuity of $\RF_L^\lambda$ (recall Lemma~\ref{l:props-gRFI} and Remark \ref{rem:finiteMeasures}) to obtain
	\begin{align}\label{eq:RF-alt}
	\MoveEqLeft\liminf_{\e \rightarrow 0} 
	\int_0^T \RF_L^\lambda(\mu_t \vert \rho_t^\e)\,dt 
	\ge \int_0^T \liminf_{\e \rightarrow 0} \RF_L^\lambda(\mu_t \vert \rho_t^\e)\,dt 
	= \int_0^T \RF_L^\lambda(\mu_t \vert \rho_t)\,dt.
	\end{align}
	Since $\e v^\e_t(x)$ is uniformly bounded for every $t\in (0,T]$ and $x\in\curlX$, we can pass $\e\rightarrow 0$ in the final term of~\eqref{ineq:FIR-intermediate} using the dominated convergence theorem, which gives
	\begin{align*}
	\lim\limits_{\e\rightarrow 0}\e\int_0^T \sum\limits_{x\in\curlX}v_t^\varepsilon(x)\partial_t\mu_t(x) dt= 0.
	\end{align*}

	Putting the results of the two steps together, we obtain
	\begin{align*}
	\frac{1}{\lambda }\I_L(\mu) 
	\ge \RelEnt(\mu_T \vert \rho_T) - \RelEnt(\mu_0 \vert \rho_0) + \int_0^T \RF_L^\lambda(\mu_t \vert \rho_t)\,dt,
	\end{align*}
	which concludes the proof of the FIR inequality.
\end{proof}

\subsection{Properties of the generalised relative Fisher Information}\label{subsec:propetiesGenRFI}

Given the  set $\curlX$ and the operator $L$, we  define a graph with vertices $\curlX$ and un-oriented edges $\curlE\subset \curlX\times \curlX$ as follows:
\[
(x,y)\in \curlE \quad\Longleftrightarrow \quad L(x,y)>0\quad\text{or}\quad L(y,x)>0.
\]
The interpretation of this graph is that two vertices are connected if they are a single jump of the Markov process apart, in either direction. 
In this graph, the support $\supp(\rho) := \{x\in \curlX : \rho(x)>0\}$ is a subset of the vertices, and defines a subgraph by deleting all edges that do not connect two vertices in $\supp(\rho)$.
Furthermore, we can decompose $\supp(\rho)$ into connected components $\Omega_i$, i.e.\ $\supp(\rho) = \cup_{i \in I} \Omega_i$ and for every pair $x,y \in \Omega_i$ there exists a finite sequence $(x_n)_{n = 1,\dots, N}$ in $\Omega_i$ with $x_1 = x$, $x_N = y$ and the vertices $x_n$ and $x_{n+1}$ are connected for all $n = 1, \dots, N-1$.

\begin{lem}\label{l:R=0} 
	Let $\mu,\rho\in\curlP(\curlX)$, and let $\supp(\rho)$ be decomposed into connected components $\Omega_i$. If $\mu=\rho$ then $\RF^\lambda_L(\mu \vert \rho)=0$. Further, if $\RF^\lambda_L(\mu \vert \rho)=0$, then there exist numbers $a_i \geq 0$, $i\in I$, such that $\mu(x) = a_i\rho(x)$ for all $x\in \Omega_i$. In particular, if $\rho(x)>0$ for all $x\in\curlX$ and $L$ is irreducible, then $\mu = \rho$.
\end{lem}
\begin{proof}
	The fact that $\mu=\rho$ implies $\RF^\lambda_L(\mu\vert\rho) = 0$ follows from the definition of  $\RF^\lambda_L$.  Assume now that $\RF_L^\lambda(\mu \vert \rho) = 0$ for $\mu, \rho \in \curlP(\curlX)$.
	Let $\Omega_i$ be a connected component of the support of $\rho$, where we exclude the trivial cases that $\mu$ vanishes identically on $\Omega_i$ or that $\Omega_i$ only contains one vertex.
	We now show that if $\mu$ does not vanish identically it is strictly positive on $\Omega_i$.
	Assume that $\mu \vert_{\Omega_i} \not> 0$; since $\Omega_i$ is a connected subgraph there exists $x,y \in \Omega_i$ such that $L(x,y) > 0$ and either $\mu(x) > 0$ and $\mu(y) = 0$ or $\mu(x) = 0$ and $\mu(y) > 0$.
	In the first case, we estimate using \eqref{eq:genRFIc} (recall that $\rho(x) > 0$ and $\rho(y) > 0$) that
	\begin{align*}
	\RF_L^\lambda(\mu \vert \rho) \geq L(x,y) \left(-\mu(x) + \dfrac{1}{\lambda}\mu(x)\right) > 0,
	\end{align*}
	since $\lambda \in (0,1)$.
	In the second case, we obtain
	\begin{align*}
	\RF_L^\lambda(\mu \vert \rho) \geq L(x,y)\left(\dfrac{\mu(y)}{\rho(y)} \rho(x)\right) > 0.
	\end{align*}
	Therefore, in both cases we obtain a contradiction to $\RF_L^\lambda(\mu \vert \rho) = 0$ and thus, $\mu \vert_{\Omega_i} > 0$.
	
	Now let $x,y \in \Omega_i$ be arbitrary.
	Since $\Omega_i$ is a connected, there exists a finite sequence $(x_n)_{n = 1,\dots,N}$ with $x_1 = x, x_N = y$ and either $L(x_n,x_{n+1}) > 0$ or $L(x_{n+1},x_n) > 0$ for all $n = 1, \dots, N-1$.
	Furthermore, $\rho > 0$ on $\Omega_i$ and thus (cf. \eqref{eq:psi_r}),
	\begin{align*}
	0 = \RF_L^\lambda(\mu \vert \rho) \geq L(x,y) \rho(x) \dfrac{r_\lambda(v(x),v(y))}{\lambda} \geq 0, \qquad v = \mu/\rho
	\end{align*}
	for all $x,y \in \Omega_i$ and hence, especially
	\begin{align*}
	r_\lambda(v(x_n), v(x_{n+1})) = 0 \text{ or } r_\lambda(v(x_{n+1}),v(x_n)) = 0.
	\end{align*}
	Using Lemma \ref{lem:r_lambda} this is true if and only if $v(x_n) = v(x_{n+1})$ and thus,
	\begin{align*}
	\dfrac{\mu(x_{n-1})}{\rho(x_{n-1})} = \dfrac{\mu(x_{n})}{\rho(x_{n})} = \dfrac{\mu(x_{n+1})}{\rho(x_{n+1})} \qquad \text{ for all } n = 2, \dots, N-1.
	\end{align*}
	Since the pair $x,y$ was arbitrarily chosen, it follows that there exists a constant $a > 0$ such that $\mu(x) = a \rho(x)$ for all $x \in \Omega_i$.
	
	Finally, if $\rho(x) > 0$ for every $x \in \curlX$ and $L$ is irreducible, then $\curlX$ itself is a connected component and we can apply the previous result.
	Furthermore, since $\mu, \rho$ have the same mass, i.e. $\mu(\curlX) = \rho(\curlX)$, we have $a = 1$ in this case.
\end{proof}

\begin{rem}
	Note that no claim is made about $\mu(x)$ for $x\not\in \supp(\rho)$; see Example~\ref{ex:genRFI} in which $\RF^\lambda_L(\mu \vert \rho)=0$, but there exist $x\in \curlX$ with $\rho(x)=0$ and $\mu(x)>0$.
	However, if one assumes additionally that $\RelEnt(\mu\vert \rho)<\infty$, then necessarily $\mu(x)=0$ for all $x\notin \supp(\rho)$. Furthermore, in the case $\mu,\rho\in\curlP(\curlX)$ with $\RelEnt(\mu\vert \rho)<\infty$, we directly recover $\mu\equiv \rho$.
\end{rem}

\begin{example}\label{ex:genRFI}
	We now give an example of $\rho,\mu$, such that $\RF_L^\lambda(\mu \vert \rho) = 0$ and $\rho(x) = 0$ but $\mu(x) > 0$ for some $x \in \curlX$. Let $w,z \in \curlX$ and $L$ such that $L(x,z)= 0$ as well as $L(z,x) = 0$ for all $x \neq w$. 
	We consider $\mu = \delta_z$ and $\rho$ with $\supp(\rho) = \curlX \setminus \{w,z\}$. The corresponding generalised relative Fisher information~\eqref{eq:genRFIc} is 
	\begin{align*}
	\RF_L^\lambda(\mu\vert\rho) =& \sum_{x,y \in \curlX \setminus \{w,z\}} L(x,y) \psi_\lambda(x,y) \\
	&+ \sum_{x \in \curlX \setminus \{w,z\}} \left[L(x,z) \psi_\lambda(x,z) + L(z,x) \psi_\lambda(z,x) + L(x,w) \psi_\lambda(x,w) + L(w,x) \psi_\lambda(w,x) \right] \\
	&+ L(w,z) \psi_\lambda(w,z) + L(z,w)\psi_\lambda(z,w).
	\end{align*}
	By the definition of $\psi_\lambda$, the first summation vanishes since 
	$\mu(x) = \mu(y) = 0$ for $x,y \in \curlX \setminus \{w,z\}$. Regarding the second summation, note that 
	$L(x,z) = L(z,x) = 0$ by assumption and thus the first two terms vanish. Furthermore, $\psi_\lambda(x,w) = 0$ since $\mu(w) = 0$ and $\psi_\lambda(w,x) = 0$ since $\rho(w) = 0$, and thus the remaining two terms vanish. The last two terms in the equality above also vanish since $\rho(w)=\rho(z) =0$.
	This show that $\RF_L^\lambda(\mu \vert \rho) = 0$ but $\mu(z) = 1 > 0$ while $\rho(z) = 0$, i.e.\ there does not exist any $a>0$ such that $\mu(x)\neq a\rho(x)$ for $x\not\in \supp(\rho)$. 
	Additionally, this gives an example for which $\mu = a \rho$ holds on a subgraph $\Omega = \curlX \setminus \{w,z\}$ with $a = 0$.
\end{example}

\medskip

\label{motivation:RFI-convergence}
Next we turn to the asymptotic behaviour of $\RF_L^\lambda$ in the limit  $\lambda\rightarrow 0$, described by Lemma~\ref{l:RFI-convergence}. Before presenting the result, we first formally derive the limit which in this case is the relative Fisher Information~\eqref{eq:RFIBasicDef}. Using~\eqref{eq:legendreTransform}, for any  $\lambda \in (0,1)$ and $f \in \Cb(\curlX)$ we find
\begin{align*}
\dfrac{1}{\lambda} \curlH(\mu, \lambda f) = \sup_{s \in \ell^1(\curlX)} \left\{\sum_{x \in \curlX} f(x) s(x) - \dfrac{1}{\lambda}\curlL(\mu,s)\right\}\geq \sum_{x \in \curlX} f(x) (L^T\mu )(x),
\end{align*}
where we have chosen $s=L^T\mu$ and used $\curlL(\mu,L^T\mu)=0$ (cf.~\eqref{eq:propertiesRateFunctional}) to arrive at the inequality.  Substituting this into~\eqref{eq:genRFIa} we arrive at 
\begin{align*}
\RF^\lambda_L(\mu \vert \rho) \leq  \sum_{x,y \in \curlX} L(x,y) \dfrac{\mu(y)}{\rho(y)}\rho(x) - \sum_{x \in \curlX} L\log\left(\dfrac{\mu}{\rho}\right)(x) \mu(x) = \RF_L(\mu \vert \rho),
\end{align*}
where $\RF_L(\cdot\vert\cdot)$ is defined in~\eqref{eq:RFIBasicDef}. Since $\curlL$ is the Lagrangian corresponding to the operator $L$, it follows that  $\curlL(\mu,s) > 0$ if $s \neq L^T\mu$ (recall the properties below~\eqref{eq:rateFunctional}). 
Hence for small $\lambda$, the deviations from $s=L^T\mu$ are penalised in the definition of the Hamiltonian~\eqref{eq:legendreTransform} and therefore for $\lambda \rightarrow 0$ we expect that the supremum is attained at $s = L^T\mu$, i.e.\ 
\begin{align*}
\lim_{\lambda \searrow 0} \dfrac{1}{\lambda} \curlH(\mu, \lambda f) = \sum_{x \in \curlX} f(x) (L^T \mu)(x) = \sum_{x,y\in \curlX} \mu(x) L(x,y) (f(y) - f(x)).
\end{align*} 
Substituting this in~\eqref{eq:genRFIa} we expect that $\RF_L^\lambda \xrightarrow{\lambda\rightarrow 0} \RF_L$. We make this intuition rigorous in the next result. 

\begin{lem}\label{l:RFI-convergence}
	$(i)$ For all $\mu, \rho \in \curlP_+(\curlX)$, $\lim_{\lambda \searrow 0} \RF_L^\lambda(\mu \vert\rho) = \RF_L(\mu \vert \rho)$ monotonically increasing.
	
	\medskip
	
	$(ii)$ $\gammalim_{\lambda \searrow 0} \RF_L^\lambda = \RF_L$ on $\curlP_+(\curlX) \times \curlP_+(\curlX)$.
\end{lem}
\begin{proof}
	$(i)$ Set $v=\mu/\rho$. Using~\eqref{eq:genRFIc},~\eqref{eq:psi_r} we find
	\begin{align*}
	\RF_L^\lambda(\mu \vert \rho) &= \sum_{x,y \in \curlX} L(x,y)\rho(x)\frac{r_\lambda(v(x),v(y))}{\lambda}.
	\end{align*}
	Using Lemma~\ref{lem:r_lambda} and applying the monotone convergence theorem we find
	\begin{align*}
	\lim_{\lambda \rightarrow 0} \RF_L^\lambda(\mu_\lambda \vert \rho_\lambda) &=  \sum_{x,y \in \curlX} L(x,y)\rho(x)\left(\lim_{\lambda \rightarrow \infty}\frac{r_\lambda(v(x),v(y))}{\lambda}\right) \\
	&= \sum_{x,y \in \curlX} L(x,y)\rho(x)\left[ v(y)-v(x)+v(x)\log\left( \frac{v(x)}{v(y)}\right)\right]
	= \RF_L(\mu \vert \rho).
	\end{align*}
	The monotonicity of the convergence follows from the monotonicity of $\lambda\mapsto \lambda^{-1}r_\lambda$ in Lemma~\ref{lem:r_lambda}.
	
	\medskip
	
	\noindent$(ii)$ The proof of the $\Gamma$-limit consists of a liminf and a limsup inequality (see~\cite[Section 1.2]{braides02} for details).
	
	The liminf inequality states that for any sequences $(\mu_\lambda)_{\lambda \geq 0},(\rho_\lambda)_{\lambda \geq 0}\subset\curlP_+(\curlX)$ which converge in $\ell^1(\curlX)$ (and therefore pointwisely) to $\mu, \rho\in\curlP_+(\curlX)$ as $\lambda \rightarrow 0$, we have
	\begin{align}
	\liminf_{\lambda \rightarrow 0} \RF_L^\lambda(\mu_\lambda \vert \rho_\lambda) \geq \RF_L(\mu \vert \rho).
	\label{eq:gen_RFI_gamma_convergence_liminf}
	\end{align}
	Using the definition~\eqref{eq:genRFIc} of $\RF_L^\lambda$, \eqref{eq:psi_r} and Lemma~\ref{lem:r_lambda}, we find with Fatou's lemma that
	\begin{align*}
	\liminf_{\lambda \rightarrow 0} \RF_L^\lambda(\mu_\lambda \vert \rho_\lambda) = \liminf_{\lambda \rightarrow 0} \sum_{x,y \in \curlX} \rho_\lambda(x) L(x,y) \frac{r_\lambda(v_\lambda(x),v_\lambda(y))}{\lambda} 
	\geq \sum_{x,y \in \curlX} \rho(x) L(x,y) \liminf_{\lambda \rightarrow 0}\frac{r_\lambda(v_\lambda(x),v_\lambda(y))}{\lambda},
	\end{align*}
	where $v_\lambda:=\mu_\lambda/\rho_\lambda$. To complete the proof of the liminf inequality~\eqref{eq:gen_RFI_gamma_convergence_liminf} we need to bound the right hand side of the inequality above by the relative Fisher Information. Setting $s_\lambda(x,y)=v_\lambda(y)/v_\lambda(x)$, we find
	\[
	\liminf_{\lambda \rightarrow 0}\frac{r_\lambda(v_\lambda(x),v_\lambda(y))}{\lambda} = v(y)-v(x) - \limsup_{\lambda \rightarrow 0} \left\{v_\lambda(x)\left(\frac{s_\lambda(x,y)^\lambda-1}{\lambda}\right) \right\}.
	\]
	Due to the pointwise convergence $v_\lambda\to v$, we have that $s_\lambda(x,y)\to s(x,y)=v(y)/v(x)$. In particular, for any $\e>0$, we find a $\lambda_\e>0$ such that $|s_\lambda(x,y)-s(x,y)|<\e$ for all $\lambda \in (0,\lambda_\e)$. 
	Consequently, $0< s_\lambda(x,y) < s(x,y) + \e$ for $\lambda \in (0,\lambda_\e)$, which yields
	\[
	\frac{s_\lambda(x,y)^\lambda-1}{\lambda} < \frac{(s(x,y)+\e)^\lambda-1}{\lambda} \qquad\text{for all\, $\lambda \in (0,\lambda_\e)$}.
	\]
	Multiplication with $v_\lambda(x)$ and passing to the limit $\lambda\to0$, we then obtain (cf.~\eqref{eq:r_lam-lim})
	\[
	\limsup_{\lambda \rightarrow 0} \left\{v_\lambda(x)\left(\frac{s_\lambda(x,y)^\lambda-1}{\lambda}\right)\right\} \le v(x)\log(s(x,y)+\e).
	\]
	Since $\e > 0$ may be chosen arbitrarily small, we obtain
	\[
	\liminf_{\lambda \rightarrow 0} \RF_L^\lambda(\mu_\lambda \vert \rho_\lambda) \ge \sum_{x,y \in \curlX} \rho(x) L(x,y)\left[v(y) - v(x) + v(x)\log\left(\frac{v(x)}{v(y)}\right)\right] = \RF_L(\mu \vert \rho),
	\]
	as required.
	
	Next we prove the limsup inequality, wherein for fixed $\mu,\rho\in\curlP_+(\curlX)$ we need to prove  the existence of a sequence $(\mu_\lambda)_{\lambda \geq 0},(\rho_\lambda)_{\lambda \geq 0}$ in $\curlP_+(\curlX)$ which satisfies 
	\begin{align*}
	\limsup_{\lambda \rightarrow 0} \RF_L^\lambda(\mu_\lambda \vert \rho_\lambda) \leq \RF_L(\mu \vert \rho). 
	\end{align*}
	Due to $(i)$ we immediately see that the constant sequence for $(\mu_\lambda)_{\lambda \geq 0},(\rho_\lambda)_{\lambda \geq 0}$, i.e.\  $\mu_\lambda = \mu$ and $\rho_\lambda = \rho$ for all $\lambda>0$ does the job, which completes the proof.
\end{proof}

\begin{rem}[Role of irreducibility]
	While from the very outset we have assumed that the generator $L$ is irreducible (cf.~\eqref{eq:conditionsGenerator3}), it is worth noting that the definition of the generalised Fisher Information~\eqref{eq:genRFIc} is well defined even when this does not hold. Furthermore the various properties of the generalised Fisher Information outlined in this and the previous section do not require irreducibility as well. However, irreducibility of the generator is required to prove the FIR inequality in Theorem~\ref{thm:FIR}. 
\end{rem}

\subsection{Modified FIR for classical relative Fisher Information}\label{sec:FIR-gFIR}

In what follows, we use the  convergence result in Lemma~\ref{l:RFI-convergence} to prove a FIR-inequality with the classical relative Fisher Information~\eqref{eq:RFIBasicDef} by restricting the class of admissible curves $\mu$. In the next result we provide sufficient conditions under which 
\[
(1-\gamma) \RF_L(\mu \vert \rho) \leq \RF^\lambda_L(\mu \vert \rho)
\]
for some $\gamma\in(0,1)$. Recall from our discussion in Section~\ref{subsec:failureClassicalFIR} that this is not true in general since we can construct a sequence for which the the relative Fisher Information is unbounded while the rate functional is bounded (and therefore the generalised Fisher Information is bounded by Theorem~\ref{thm:FIR}). In fact, from Lemma~\ref{l:RFI-convergence} we know that the generalised Fisher Information $\RF_L^\lambda$ is always bounded from above by the Fisher Information $\RF_L$, and in the following result we show that the inequality can be reversed under certain conditions. 
\begin{lem}\label{lem:FIRrestrictedMeasure}
	Fix $K < \infty$, $\lambda \in (0,1)$ and let $\mu, \rho \in \curlP_+(\curlX)$ satisfy 
	\begin{align*}
	\sup_{x \in \curlX} \snormlr{\log\left(\dfrac{\mu(x)}{\rho(x)}\right)} \leq K.
	\end{align*}
	Then there exists a $\gamma = \gamma(K, \lambda) >0$ such that
	\begin{align}
	(1-\gamma) \RF_L(\mu \vert \rho) \leq \RF^\lambda_L(\mu \vert \rho). \label{eq:connectionRFIgenRFI}
	\end{align}
	Furthermore for every $K < \infty$ there exists a $\lambda_0 \in (0,1)$ such that $\gamma(K,\lambda) \in (0,1)$ for all $\lambda \in (0,\lambda_0)$.
\end{lem}
\begin{proof}
	The uniform bound on the logarithm implies that  $\RF_L(\mu \vert \rho)$ is well-defined. Using the definitions of these objects we can rewrite~\eqref{eq:connectionRFIgenRFI}  as
	\begin{align*}
	\dfrac{1}{\lambda} \curlH\left(\mu, \lambda\log\left(\dfrac{\mu }{\rho}\right)\right) - \sum_{x,y \in \curlX} \mu(x) L(x,y) \log\left(\dfrac{\mu(y) \rho(x)}{\rho(y) \mu(x)}\right)	= \RF_L(\mu \vert \rho) - \RF_L^\lambda(\mu \vert \rho)  \\
	\leq \gamma \RF_L(\mu \vert \rho) = \gamma\left[\curlH\left(\mu, \log\left(\dfrac{\mu }{\rho}\right)\right) - \sum_{x,y \in \curlX} \mu(x) L(x,y) \log\left(\dfrac{\mu(y) \rho(x)}{\rho(y) \mu(x)}\right)\right].
	\end{align*}
	To simplify the notation, we define
	\begin{align*}
	\mathcal{D}(\mu, f) := \curlH(\mu, f) - \sum_{x,y \in \curlX} \mu(x) L(x,y) (f(y) - f(x)) = \sum_{x,y \in \curlX} \mu(x) L(x,y) \left[e^{\nabla f(y,x)} - (1 + \nabla f(y,x))\right],
	\end{align*}
	where $\nabla f(y,x) = f(y) - f(x)$.
	Using the Taylor expansion of the exponential, we estimate 
	\begin{align*}
	&\mathcal{D}(\mu, \lambda f)  \leq \sum_{x,y \in \curlX} \mu(x) L(x,y)\sum\limits_{n\geq 2} \lambda^n\frac{|\nabla f(y,x)|^n}{n!}= \lambda^2 \sum_{x,y \in \curlX} \mu(x) L(x,y)\sum\limits_{n\geq 2} \lambda^{n-2}\frac{|\nabla f(y,x)|^n}{n!}\\
	&\leq \lambda^2 \sum_{x,y \in \curlX} \mu(x) L(x,y)\sum\limits_{n\geq 2}\frac{|\nabla f(y,x)|^n}{n!}
	=\lambda^2 \sum_{x,y \in \curlX} \mu(x) L(x,y) \left[e^{\snorm{\nabla f(y,x)}} - (1 + \snorm{\nabla f(y,x)})\right] =: \lambda^2 \tilde{\mathcal{D}}(\mu, f),
	\end{align*}
	where the second inequality follows since $\lambda \in(0,1)$. 
	Next, we show that there exists a $c_K > 0$ only depending on $K$ such that $\mathcal{D}(\mu, f) \ge c_K \tilde{\mathcal{D}}(\mu, f)$ uniformly for all $f$ with $\norm[\infty]{f} \leq K$. This is equivalent to proving that 
	\begin{align*}
	\varphi(\alpha) := \dfrac{e^\alpha - (1+\alpha)}{e^{\snorm{\alpha}} - (1 + \snorm{\alpha})} \geq c_K
	\end{align*}
	for $\alpha \in [-2K, 2K]$.
	If $\alpha > 0$, then $\varphi(\alpha) = 1$ and hence, it is sufficient to consider $\alpha \leq 0$.
	By using l'Hospital, we can continuously extend $\varphi$ to $\alpha = 0$ by defining $\varphi(0) =1$.
	Furthermore, $\varphi$ is positive and monotonically decreasing for $\alpha < 0$.
	Since $[-2K,2K]$ is compact, the existence of $c_K>0$ follows from the continuity and positivity of $\varphi$.
	
	We thus established that for every $K < \infty$, there exists a $c_K > 0$ only depending on $K$ such that
	\begin{align*}
	\dfrac{1}{\lambda} \mathcal{D}\left(\mu, \lambda \log\left(\dfrac{\mu }{\rho}\right)\right) \leq \dfrac{\lambda}{c_K} \mathcal{D}\left(\mu, \log\left(\dfrac{\mu }{\rho}\right)\right).
	\end{align*}
	Choosing $\gamma = \lambda/c_K>0$ then yields~\eqref{eq:connectionRFIgenRFI} and for all $\lambda < c_K$, we obtain $\gamma \in (0,1)$.
\end{proof}

Using this result along with Theorem~\ref{thm:FIR} we arrive at a modified FIR inequality for the classical relative Fisher Information. 
\begin{prop}\label{prop:usualFIR}
	Let $\rho\in A\curlC([0,T];\curlP(\curlX))$ be a solution of \eqref{eq:forwardKolmogorovEquation} and $\mu \in \curlC([0,T];\curlP_+(\curlX))$ satisfy $\I_L(\mu) + \RelEnt(\mu_0\vert\rho_0)<\infty$.
	Furthermore assume that there exists a $K < \infty$ such that
	\begin{align*}
	\sup_{t \in [0,T]} \sup_{x \in \curlX} \snormlr{\log\left(\dfrac{\mu_t(x)}{\rho_t(x)}\right)} \leq K. 
	\end{align*}
	Then there exists a sufficiently small $\lambda$ (see Lemma~\ref{lem:FIRrestrictedMeasure}) such that 
	\begin{align*}
	\RelEnt(\mu_T \vert\rho_T)  + (1- \gamma) \int_0^T \RF_L(\mu_t \vert \rho_t) \,dt \leq  \RelEnt(\mu_0 \vert \rho_0)+ \dfrac{1}{\lambda} \I_L(\mu),
	\end{align*}
	with $\gamma \in (0,1)$.
\end{prop}

\begin{rem}[Convexity of generalised Fisher Information]
	Let $\mu,\rho\in\curlP_+(\curlX)$. Using the explicit representation for the Hamiltonian~\eqref{eq:hamiltonian} we find
	\begin{align*}
	\RF_L^\lambda(\mu \vert \rho) &= \sum_{x, y \in \curlX} L(x,y) \left[\mu(y)\dfrac{\rho(x)}{\rho(y)} - \mu(x)\right] - \dfrac{1}{\lambda} \sum_{x,y \in \curlX} \mu(x) L(x,y) \left[\left(\dfrac{\mu(y)\rho(x)}{\mu(x)\rho(y)}\right)^\lambda  - 1\right]\\
	&=\sum_{x, y \in \curlX} L(x,y) \left[\mu(y)\dfrac{\rho(x)}{\rho(y)} - \mu(x)\right] - \dfrac{1}{\lambda} \sum_{x,y \in \curlX} L(x,y) \left[\left(\dfrac{\rho(x)}{\rho(y)}\right)^\lambda \mu(y)^\lambda \mu(x)^{1-\lambda}  - \mu(x)\right]
	\end{align*}
	Since $\alpha^\lambda \beta^{1-\lambda}$ is concave for $\alpha,\beta>0$ and $\lambda\in (0,1)$ it follows that the third term on the right hand side is concave in $\mu$. Since the rest of the terms on the right hand side are linear in $\mu$ it follows that the generalised Fisher Information is convex in the first entry. 
\end{rem}

\subsection{Comparison with diffusion processes}\label{sec:ContStateSpace}

So far we have limited our discussion to Markov jump processes. In this section we will apply the connections between the relative entropy, the generalised Fisher Information and the rate functional described earlier to the case of diffusions. In what comes next, we first define each of these objects for diffusions and then connect to  the existing literature. Since our focus in this paper is on the discrete setting, we will keep the treatment in this section formal. 

Consider a stochastic differential equation on $\R^d$,
\begin{align}\label{eq:DiffSDE}
dX_t=b(X_t)dt + \sqrt{2} \sigma(X_t)dB_t,
\end{align}
where $b:\R^d\rightarrow\R^d$, $\sigma:\R^d\rightarrow\R^{d\times d}$, $B_t$ is a standard Brownian motion in $\R^d$ and $X_0\in\R^d$ is the initial data. The corresponding forward Kolmogorov equation (also called the Fokker-Planck equation in this case) evolves  according to 
\begin{align}\label{eq:FP}
\begin{dcases}	
\partial_t \rho = L^T \rho := \operatorname{div}(b \rho) + \nabla^2 : A\rho \\
\rho_{t = 0} = \rho_0,
\end{dcases}
\end{align}
where $A:=\sigma\sigma^T\in\R^{d\times d}$, $\rho_0\in \curlP(\R^d)$ is the initial data and $\nabla^2$ is the Hessian. Here $L^T$ is the adjoint corresponding to the generator 
\begin{align}\label{eq:diff-gen}
Lf(x):=- b(x) \cdot \nabla f(x) + A(x) : \nabla^2 f(x).
\end{align}
Throughout this section we assume that the coefficients and the solution to~\eqref{eq:FP} are sufficiently smooth (for a more general setup see~\cite{DLPSS-TMP}). 
For  any probability measures $\mu,\rho\in\curlP(\R^d)$ and a Markov generator $L$, we define the relative Fisher Information as
\begin{align*}
\RF_L(\mu\vert \rho) := \int_{\R^d}\left[ - L\log\left(\dfrac{\mu}{\rho}\right)\mu + L\left(\dfrac{\mu}{\rho}\right)\rho\right]= \int_{\R^d}\left| \nabla\log\left(\dfrac{\mu}{\rho}\right)\right|^2_A\mu,
\end{align*}
where $|x|^2_A:=x^TAx$. This is continuous version of the classical relative Fisher Information~\eqref{eq:RFIBasicDef}. Here we have inherently assumed that $\mu,\rho$ have sufficiently smooth densities (not renamed) such that this object is well defined. Note that, since we are working with `linear' diffusion processes, the Fisher Information depends on the generator $L$ only via the matrix $A$.  As in the discrete case (recall~\eqref{eq:Ent-Fish-Rel}), when $\mu,\rho$ are solutions to~\eqref{eq:FP}, the relative Fisher Information satisfies the relation 
\begin{align*}
\RF_L(\mu_t \vert \rho_t) = - \dfrac{d}{dt} \RelEnt(\mu_t \vert \rho_t).
\end{align*}
The corresponding large-deviation rate functional $\I_{L}:\curlC([0,T];\curlP(\curlX))\rightarrow\R$ is (see eg.~\cite{dawsonGaertner87,Oelschlager84}) 
\begin{align}\label{def:RF-diff}
\I_L(\mu) = \sup_{f \in C^{1}([0,T];C^2_b(\R^d))}  \int_{\R^d} f_T \,d\mu_T - \int_{\R^d} f_0 \,d\mu_0 - \int_0^T \left(\int_{\R^d} \partial_t f \,d\mu_t   + \curlH(\mu_t, f_t)\right) \,dt,
\end{align}
with the Hamiltonian
\begin{align}\label{eq:Diff-Ham}
\curlH(\mu, f) :=\int_{\R^d} e^{-f} L e^{f} \,d\mu  = \int_{\R^d} Lf + \Gamma(f,f) \,d\mu.
\end{align}
Here $\Gamma$ is the carr{\'e}-du-champ operator corresponding to the Markov generator $L$ (see~\cite[Section 1.4.2]{bakry14}) 
\begin{align*}
\Gamma(f,g) := \frac{1}{2}\left[L(fg)-f Lg - gLf\right] = \nabla f\cdot A\nabla g.
\end{align*}
The ($A$-weighted) quadratic structure on the right hand side is particular to the diffusion processes.

For any $\lambda\in(0,1)$, and probability measures $\mu,\rho\in\curlP_+(\R^d)$, the continuous state-space counterpart of the generalised Fisher Information~\eqref{eq:genRFI} is
\begin{align*}
\RF^\lambda_L(\mu \vert \rho) := \int_{\R^d}\frac{\mu}{\rho}L^*\rho -  \dfrac{1}{\lambda} \curlH\left(\mu, \lambda \log\left(\dfrac{\mu }{\rho}\right)\right) =  \left(1 - \lambda\right) \RF_L(\mu \vert \rho),
\end{align*}
where $L^*$ denotes the $L^2(\R^d,\rho)$-adjoint of $L$.
The equality here follows by using~\eqref{eq:Diff-Ham}. Note that this is different from the discrete case where the generalised Fisher Information is bounded from above by the relative Fisher Information (recall Lemma~\ref{l:RFI-convergence}) and the reversed inequality only holds in a fairly restrictive setting (see Lemma~\ref{lem:FIRrestrictedMeasure}). This is due to the simpler structure of the Hamiltonian~\eqref{eq:Diff-Ham} which can be written as a combination of a linear and a quadratic term, as opposed to a genuine exponential structure in the discrete case.

Following the formal approach used for deriving the FIR inequality (cf.~Section~\ref{sec:MainRes}), we arrive at
\begin{align*}
\RelEnt(\mu_T \vert \rho_T) + (1-\lambda)\int_0^T\RF_L(\mu_t \vert \rho_t)dt \leq   \RelEnt(\mu_0 \vert \rho_0) + \dfrac{1}{\lambda} \I_L(\mu),
\end{align*}
which has been derived recently in~\cite{DLPSS-TMP}, and without the connection to large deviations in~\cite{bogachevRoecknerShaposhnikov16}. In~\cite{bogachevRoecknerShaposhnikov16} such an inequality is proven rigorously by directly studying the time derivative of the relative entropy and using appropriate regularity results for a very wide class of Fokker-Planck equations, while here we derive this inequality by studying the dual formulation of the rate functional. Similar ideas have also been developed for the (nonlinear) Vlasov-Fokker-Planck equation in~\cite[Theorem 2.3]{duongLamaczPeletierSharma17}.

\section{Coarse-graining }\label{sec:CGApplication}

In this section we study the coarse-graining problem introduced in  Section~\ref{subsec:application}, which we now recall. 
Consider a family of forward Kolmogorov equations
\begin{align}
\begin{dcases}
\partial_t \mu^\e = (L^\e)^T \mu^\e, \\
\mu^\e_{t = 0} = \mu_0,
\end{dcases}
\label{eq:KolEq}
\end{align}
on $\curlX = \curlY \times \curlZ$ with $\curlY = \{0,1\}$ and $\curlZ = \{1, \dots, n\}$, generated by the family of operators 
\begin{align}
L^\e= \frac{1}{\e}Q +  C := \frac{1}{\e}\begin{pmatrix}Q_0 & 0 \\ 0 & Q_1\end{pmatrix} +  \begin{pmatrix}D_0 & C_{0,1} \\
C_{1,0} & D_1 \end{pmatrix},
\label{eq:CG-res-gen}
\end{align}
i.e.\ with
\[
Q((y,z),(y',z')) = \begin{cases}
Q_y(z,z') &\text{if\, $y'=y$}\\
0 &\text{otherwise}
\end{cases},\qquad C((y,z),(y',z')) = \begin{cases}
C_{y,y'}(z,z') &\text{if\, $y'\ne y$}\\
D_y(z) & \text{if $y' = y$ and $z'=z$} \\ 
0 &\text{otherwise}
\end{cases}
\]
for $x=(y,z)$, $x'=(y',z')\in\curlX$ satisfying
\begin{align*}
\forall x\in\curlX: \ \sum_{x'\in\curlX}Q(x,x')=0=\sum_{x\in\curlX} C(x,x'),
\end{align*}
and diagonal matrix $D_y$, $y\in\curlY$, which satisfies 
\begin{align}\label{eq-def:D}
\forall z\in\curlZ: \ D_y(z):=-\sum_{z'\in\curlZ} C_{y,1-y}(z,z').
\end{align}
Here $L^\e$ is irreducible, and therefore~\eqref{eq:KolEq} admits a stationary solution $\pi^\e\in\curlP(\curlX)$. Additionally we assume that $Q_0$ and $Q_1$ are irreducible as well.
In what follows we will use $\nabla f(y,x):=f(y)-f(x)$. 

\begin{rem}[Topologies on $\curlP(\curlX)$]\label{rem:topologies}
	Since $\curlX$ is a finite set, $\curlP(\curlX)$ can be identified with a closed, bounded (and thus compact) subset of the finite-dimensional vector space $\R^\curlX$.
	Therefore, there is no necessity to distinguish between different notions of convergence on $\curlP(\curlX)$, since there is a unique topology which makes $\R^\curlX$ a (Hausdorff) topological vector space.
	In particular, the notion of uniform convergence (generated by the total variation distance) and narrow convergence (weak convergence with test functions in $C_b$) are equivalent and coincide with the standard convergence on $\R^\curlX$.
\end{rem}

The rest of this section is devoted to studying the behaviour of~\eqref{eq:KolEq} in the limit of $\e\rightarrow 0$.  We now outline an abstract variational framework, developed in \cite{duongLamaczPeletierSharma17}, that will be used to study this problem. 

\subsection{A variational framework for coarse-graining}\label{subsec:variationalFrameworkCG}
Let $\rho^\e:[0,T]\rightarrow \curlP(\curlX)$ be a family of solutions to the forward Kolmogorov equations~\eqref{eq:KolEq}, and let $\I_{L^\e}$ be the corresponding family of large-deviation rate functionals associated to the underlying stochastic process (recall Theorem~\ref{thm:LDP}). 
Since the solutions $\rho^\e$ is characterised by $\I_{L^\e}$ via $\I_{L^\e}(\rho^\e)=0$, establishing the limit behaviour as $\e\rightarrow 0$ consists of answering two questions:

\begin{enumerate}[label=(\arabic*)]
	\item \textit{Compactness:} Do solutions of $\I_{L^\e}(\rho^\e) = 0$ have useful compactness properties, allowing one to extract a subsequence that converges in a suitable topology, say $\tau$?
	
	\item \textit{Liminf inequality:} Is there a limit functional $\I\geq 0$ such that 
	\begin{align}\label{que-LimInf}
	\rho^\e\xrightarrow{\tau} \rho \  \Longrightarrow \ \liminf_{\e \searrow 0} \I_{L^\e}(\rho^\e) \geq \I(\rho)?
	\end{align}
	And if so, does one have 
	\begin{align*}
	\I(\rho)=0 \ \Longleftrightarrow \ \partial_t\rho= L^T\rho,
	\end{align*}
	for some limiting operator $L$?
\end{enumerate}

As we shall see in the coming sections, the method we use answers both these questions for \emph{approximate solutions}. By this we mean that we work with a sequence of time-dependent probability measures which satisfy $\sup_{\e>0} \I_{L^\e}(\mu^\e)<\infty$. The exact solutions are special cases when $\I_{L^\e}(\mu^\e)=0$. Consequently, all our results follow from this uniform bound and assumptions on well-prepared initial data (which is exactly the right hand side of the FIR inequality~\eqref{eq:FIRwGeneralisedRFI}). 

The question of compactness will be answered by the uniform bound on the rate functional. Since our state space is finite, this bound along with the Arzel{\`a}-Ascoli theorem will provide us with suitable compactness properties (see Section~\ref{sec:CGComp} for details).  

In answering the second question, we will make use of two crucial ingredients. First, that the rate functional has a duality relation of the type (recall Theorem~\ref{thm:LDP}), 
\begin{align}\label{def:RF-dual}
\I_L(\mu) = \sup_{f} \LDJ_L(\mu, f),
\end{align}
where the supremum is taken over an appropriate class of functions. Second, that the problem is of coarse-graining type as we expect that in the limit of $\e\rightarrow 0$, the dynamics in each macro-state equilibrates and the limiting object is a jump process across the macro-states (recall discussion in Section~\ref{subsec:application}). We characterise this behaviour by means of a coarse-graining map which identifies the relevant degrees of freedom. In our setting we choose this to be a mapping onto the macro-states, i.e.\ $\xi:\curlX\rightarrow\curlY$ with $\xi(x)=y$ for every $x=(y,z)\in\curlX$. The coarse-grained equivalent of $\rho^\e:[0,T]\rightarrow\curlP(\curlX)$ is the push-forward $\hat\rho^\e:=\xi_\#\rho^\e:[0,T]\rightarrow\curlP(\curlY)$. 
For a discussion on coarse-graining mappings in other contexts see~\cite[Section 1.4]{Sharma17}. 

The core of the argument for the liminf inequality~\eqref{que-LimInf} is summarised in the following formal calculation:
\begin{eqnarray}
\I_{L^\e} (\rho^\e) &=& \sup_f \; \LDJ_{L^\e}(\rho^\e,f)\nonumber\\
&\stackrel{f=g\circ \xi} \geq & \sup_g \;\LDJ_{L^\e}(\rho^\e,g\circ \xi)\nonumber\\
&& \phantom{\sup\; \widehat \LDJ_{L^\e}(}\Big\downarrow\;\e\rightarrow 0 \label{eq:formalLiminf}\\[\jot]
&&\sup_g\;  {\LDJ}({\rho},g\circ\xi)\nonumber\\
&\stackrel{(\ast)}\eqqcolon & \sup_g\; {\hat{\LDJ}}({\hat{\rho}},g)\quad
\stackrel{(\ast\ast)}\eqqcolon \quad{\hat{\I}}({\hat{\rho}}) \nonumber
\end{eqnarray}
Let us now go through each of these lines. The first line is the dual characterisation of the rate functional~\eqref{def:RF-dual}. The  inequality on the second line follows by restricting the class of admissible functions $f$ to functions of the type $f = g \circ \xi$. Here we have made a choice to restrict ourselves to functions of the form $f=g\circ\xi$. Following this inequality we pass to the limit using the compactness results derived earlier.  The choice of coarse-graining map is crucial here since we cannot expect convergence for functions $f$ which still have access to the full information. 

In the next step $(\ast)$, we pass from the full limit measure $\rho$ to the coarse-grained measure $\hat{\rho}$.
To do that rigorously we need a \emph{local-equilibrium} result, which describes how we can reconstruct the full information in $\rho$ which is lost by considering only $\hat{\rho}$. As we shall see in Section~\ref{sec:LocEq}, this result crucially depends on the generalised Fisher Information and the FIR inequality. 

Finally, we define in $(\ast\ast)$ a new functional $\hat{\I}$. In a successful application of coarse-graining, this functional is connected to an evolution equation similar to~\eqref{eq:propertiesRateFunctional}.
In our example it turns out that $\hat{\I}$ is again a large deviations rate functional and connected to a lower dimensional effective equation.

In what follows we go through each of the steps described above to derive the behaviour of~\eqref{eq:KolEq} as $\e\rightarrow 0$. In Section~\ref{sec:CGComp} we prove compactness results, Section~\ref{sec:LocEq} contains the local-equilibrium result and in Section~\ref{sec:CGLimInf} we prove the liminf inequality.

\subsection{Compactness}\label{sec:CGComp}
In the following result we discuss the compactness properties. We prove a two-level compactness result, a weaker result on the original space $\curlX$ and a stronger result on the coarse-grained space $\curlY$. 

\begin{lem}\label{lem:compactness}
	Let a sequence $\mu^\e\in \curlC([0,T];\curlP(\curlX))$ satisfy 
	\[
	\sup_{\e >0} \I_{L^\e}(\mu^\e)<\infty.
	\] 
	Then there exists $\mu\in \curlM([0,T] \times \curlX)$ and a subsequence (not relabelled) such that 
	\begin{enumerate}[label=(\roman{*})]
		\item $\mu^\e \rightarrow \mu$ in $\curlM([0,T] \times \curlX)$ narrowly with $\mu=\int_0^T \mu_t$ for a Borel family $\{\mu_t\}_{t\in(0,T)}$.
		\item $\xi_{\#}\mu^\e \rightarrow \xi_{\#}\mu$ in $\curlC([0,T];\curlP(\curlY))$ with respect to the uniform topology in time.
	\end{enumerate}
\end{lem}
\begin{proof}
	Since $[0,T] \times \curlX$ is compact, every subset of $\curlM([0,T]\times \curlX)$ is tight. Furthermore, since $\mu^\e = \int_0^T \mu^\e_t$ with $\mu^\e_t \in \curlP(\curlX)$ the set $\{\mu^\e, \e > 0\}$ is uniformly bounded in $\curlM([0,T]\times \curlX)$, and so by Prokhorov's theorem and the equi-integrability of the map $t\mapsto \mu_t^\e(\curlX)$, we have that $\mu^\e\to\mu$ narrowly in $\curlM([0,T]\times \curlX)$ for some $\mu\in \curlM([0,T]\times \curlX)$. Furthermore where $\mu$ has the representation $\mu = \int_0^T \mu_t$ for a Borel family $\{\mu_t\}_{t\in(0,T)}$ due to the disintegration theorem.
	
	To prove the second statement we use the Arzel{\`a}-Ascoli theorem~\cite[Theorem 45.4]{munkres00}. 
	Using the characterisation \eqref{eq:BoundRF} of the rate functionals $\I_{L^\e}$, we obtain 
	\begin{align}\label{eq:RF-ModLag}
	M \ge \I_{L^\e}(\mu^\e) \ge \int_0^T \left[\left\langle \Indicator_{[s_1,s_2]}(t)\frac{g\circ\xi}{\lambda}, \partial_t \mu_t^\e\right\rangle - \curlH^\e\!\left(\mu_t^\e, \Indicator_{[s_1,s_2]}(t)\frac{g\circ\xi}{\lambda}\right) \right]\,dt,
	\end{align}
	for any $s_1,s_2\in[0,T]$, $g\in\Cb(\curlY)$ and $\lambda>0$, where $\curlH^\e$ is the Hamiltonian corresponding to the generator $L^\e$ (see~\eqref{eq:hamiltonian}). We then calculate
	\begin{align*}
	\curlH^\e\!\left(\mu_t^\e, \Indicator_{[s_1,s_2]}(t)\dfrac{g \circ \xi}{\lambda}\right) &= \sum_{x_1 \in \curlX} \mu_t^\e(x_1) \sum_{z_2 \in \curlZ} \e^{-1} Q_{y_1}(z_1,z_2) \left(e^{-\frac{1}{\lambda}\Indicator_{[s_1,s_2]}(t)\nabla g(y_1,y_1)}-1\right) \\
	&\hspace*{5em}+ \sum_{x_1 \in \curlX} \mu_t^\e(x_1) \sum_{z_2 \in \curlZ} C_{y_1,1-y_1}(z_1,z_2) \left(e^{-\frac{1}{\lambda}\Indicator_{[s_1,s_2]}(t)\nabla g(y_1,1-y_1)}-1\right) \\
	&\leq 0+ \bar{C} \left(e^{\frac{1}{\lambda}2\norm[\infty]{g}}-1\right)\Indicator_{[s_1,s_2]}(t),
	\end{align*}
	where $\bar{C}:=\sup_{y\in\curlY}\|C_{y,1-y}\|$ is independent of $\e > 0$ and $s \in [0,T]$ and the zero in the final inequality follows since $\nabla g(y_1,y_1) = 0$. Note that $D_0$ and $D_1$ do not contribute to the equality above. Substituting this bound into \eqref{eq:RF-ModLag} with $\lambda = -\|g\|_\infty/\log\sqrt{|s_2-s_1|}$ and using absolute continuity on $t \mapsto \mu_t^\varepsilon$ we find
	\begin{align*}
	\langle g, \xi_\#\mu_{s_2}^\e - \xi_\#\mu_{s_1}^\e\rangle = \int_{s_1}^{s_2} \langle g\circ\xi, \partial_t \mu_t^\e\rangle \,dt &\le \lambda M + \lambda\bar{C}|s_2-s_1|\left(e^{\frac{1}{\lambda}2\norm[\infty]{g}}-1\right) \\
	&= \frac{\|g\|_\infty M}{-\log\sqrt{|s_2-s_1|}} + \frac{\|g\|_\infty \bar{C}|s_2-s_1|}{-\log\sqrt{|s_2-s_1|}}\left(
	\frac{1}{|s_2-s_1|}-1\right)\\
	&\le 2\|g\|_\infty\frac{M + \bar{C}|1-|s_2-s_1||}{|\log |s_2-s_1||}.
	\end{align*}
	Since the narrow topology coincides with the uniform topology and the upper bound does not depend on $\varepsilon$ this gives equicontinuity of $(\xi_\#\mu^\e)$.
	Furthermore, $\xi_\#\mu^e$ is naturally bounded from above in $\curlC([0,T];\curlP(\curlY))$ and thus, we can apply the Arzel{\`a}-Ascoli theorem which gives the statement.
\end{proof}

\subsection{Local-equilibrium}\label{sec:LocEq}
As stated earlier, our interest is in studying the slow behaviour of the dynamics and we do this by focussing on a coarse-grained description of the model (via $\xi$). However information is lost in the coarse-graining procedure, and in this section we reconstruct this lost information by proving a `local-equilibrium' result, which crucially depends on the FIR inequality. 

The central idea is to pass $\e\rightarrow 0$ in the FIR inequality, obtain a vanishing bound on the generalised Fisher Information and then study the properties of the limiting object. 
More precisely, we combine the lower-semicontinuity property of $\RF_L^\lambda$ with the FIR inequality \eqref{eq:FIRwGeneralisedRFI} to show that in the limit of $\e\rightarrow 0$, the time-dependent sequence $\mu^\e$ becomes stationary in the micro-state variable and the time dependence completely shifts onto the macro-state variable. We first prove an auxiliary lemma which discusses the limit of the stationary measure $\pi^\e$ and then prove the local-equilibrium result.  
\begin{lem}\label{lem:stationary-measures}
	Let $(\pi^\e)_{\e > 0}\subset\curlP(\curlX)$ be a sequence of stationary measures corresponding to $L^\e$, i.e.\ $(L^\e)^T\pi^\e = 0$ for every $\e>0$. Then there exists a positive probability measure $\pi\in\curlP_+(\curlX)$ satisfying $Q^T\pi=0$, with $\pi^\e\to\pi$ in $\curlP_+(\curlX)$.
\end{lem}
\begin{proof}
	Due to the compactness of $\curlP(\curlX)$, we find some $\pi\in\curlP(\curlX)$ such that $\pi^\e\rightarrow\pi$ as $\e\rightarrow 0$. Passing $\e\rightarrow0$ in $\e(L^\e)^T\pi^\e=0$ yields
	\begin{align}\label{eq:LimitStatMeas}
	Q^T\pi=0  \ \Longleftrightarrow  \ \exists\, \alpha\in [0,1] \text{ such that } \pi=\begin{pmatrix}\alpha\pi_0\\ (1-\alpha)\pi_1\end{pmatrix},
	\end{align}
	where $\pi_y\in\curlP(\curlZ)$ is the stationary measure of $Q_y$, $y\in\curlY$. 
	
	We now show that $\pi\in\curlP_+(\curlX)$, which follows if $\alpha\in(0,1)$ since $\pi_y\in\curlP_+(\curlZ)$ due to the irreducibility of $Q_y$. Using $(L^\e)^T\pi^\e=0$ and $\sum_{z'\in \curlZ} Q_{y}(z,z')=0$, for every $y\in\curlY$ we find
	\begin{align*}
	0=\sum_{z\in \curlZ}((L^\e)^T\pi^\e)(y,z)&=\sum_{z,z'\in\curlZ}\Bigr[\frac{1}{\e} Q_y(z',z)\pi^\e(y,z')+C_{1-y,y}(z',z)\pi^\e(1-y,z')\Bigl]+\sum_{z\in\curlZ}D_y(z)\pi^\e(y,z)\\
	&= \sum_{z,z'\in\curlZ}C_{1-y,y}(z',z)\pi^\e(1-y,z') + \sum_{z\in\curlZ} D_y(z)\pi^\e(y,z),\end{align*}
	Furthermore passing $\e\rightarrow 0$ and using~\eqref{eq-def:D} we obtain
	\begin{align*}
	0=-\sum_{z\in\curlZ}D_{1-y}(z)\pi(1-y,z) + \sum_{z\in\curlZ}D_y(z)\pi(y,z).
	\end{align*}
	Finally, using~\eqref{eq:LimitStatMeas} and $\lambda_y:=-\sum_{z\in\curlZ}D_y(z)\pi_y(z)$ we have
	\begin{align*}
	-\alpha\lambda_0+(1-\alpha)\lambda_1=0  \ \Longrightarrow \ \alpha=\frac{\lambda_1}{\lambda_0+\lambda_1}.
	\end{align*}
	Since $\lambda_y>0$ (recall that $L^\varepsilon$ is irreducible if and only if $C_{y,1-y}$ has at least one positive entry for all $y \in \{0,1\}$) we have $\alpha\in(0,1)$ and therefore $\pi\in\curlP_+(\curlX)$. 
\end{proof}

\begin{lem}\label{lem:localEquilibrium}
	Let a sequence $\mu^\e\in \curlC([0,T];\curlP(\curlX))$ satisfy 
	\begin{align}
	\sup_{\e >0} \Big\{\I_{L^\e}(\mu^\e) + \RelEnt(\mu_0^\e\vert\pi^\e) \Bigr\}< \infty,
	\label{eq:locEquiBounds}
	\end{align}
	where $(\pi^\e)_{\e > 0}\subset\curlP(\curlX)$ is a sequence of stationary measures of $L^\e$ converging to $\pi\in \curlP_+(\curlX)$ as $\e \rightarrow 0$.
	Then there $\hat{\mu} \in \curlC([0,T];\curlP(\curlY))$ such that for almost every $t \in [0,T]$,
	\begin{align}
	\forall y\in\curlY, \ A_\curlZ \subset \curlZ, \ \mu_t(\{y\} \times A_\curlZ ) = \hat{\mu}_t({y}) \pi_y(A_\curlZ).
	\label{eq:CGLimitDecomposition}
	\end{align}
	Here $\mu$ is the limit of $(\mu^\e)_{\e>0}$ (see Lemma \ref{lem:compactness}) and for each $y\in\curlY$, $\pi_y \in \curlP(\curlZ)$ is the stationary measure corresponding to $Q_y$.
	Furthermore $\xi_\# \mu^\e \rightarrow \hat{\mu}$ in $\curlC([0,T];\curlP(\curlY))$ uniformly in time.
\end{lem}
\begin{proof}

	Using \eqref{eq:locEquiBounds} and the FIR inequality in Theorem \ref{thm:FIR}, we find
	\begin{align*}
	\RelEnt(\mu^\e_T \vert \pi^\e) + \int_0^T \RF_{L^\e}^\lambda(\mu^\e_t \vert \pi^\e) \,dt \leq \I_{L^\e}(\mu^\e) + \RelEnt(\mu^\e_0 \vert \pi^\e) \leq M \quad \Longrightarrow \quad \int_0^T \RF_{L^\e}^\lambda(\mu^\e_t \vert \pi^\e) \,dt \leq M,
	\end{align*}
	for some constant $M < \infty$ independent of $\e$.
	Recall that $L^\e = \e^{-1} Q + C$. Due to the linearity of $\RF_L^\lambda$ with respect to $L$, we find that
	\begin{align*}
	\e^{-1} \int_0^T \RF_Q^\lambda(\mu_t^\e \vert \pi^\e) \,dt + \int_0^T \RF_C^\lambda(\mu_t^\e \vert \pi^\e) \,dt \leq M.
	\end{align*}
	Multiplying with $\e$ and letting $\e \rightarrow 0$ we find 
	\begin{align*}
	\liminf_{\e \rightarrow 0} \int_0^T \RF^\lambda_Q(\mu_t^\e \vert \pi^\e) \,dt \leq 0.
	\end{align*}
	Using the non-negativity and lower-semicontinuity property of the generalized relative Fisher Information (cf.~Lemma~\ref{l:props-gRFI}), together with the Borel-measurability of the non-negative functions $t\mapsto \RF_Q(\mu_t^\e \vert \pi^\e)$, we obtain from Fatou's lemma that
	\begin{align}
	\RF_Q^\lambda(\mu_t \vert \pi) = 0\qquad\text{for almost every $t\in(0,T)$}.
	\label{eq:locEqRelRFIvanishes}
	\end{align}
	
	In what follows, for $y\in \curlY$ we use $\mu_t(\cdot|y)\in\curlP(\curlZ)$ for the family of conditional measures corresponding to $\mu_t$, i.e.\ we write $\mu_t(y,z)=\mu_t(z\vert y)(\xi_{\#}\mu_t)(y)$. We show that $\RF_Q^\lambda(\mu_t \vert \pi) = 0$ if and only if $\mu_t(z \vert y) = \pi_y(z)$ for any $x = (y,z) \in \curlX$ with $(\xi_\# \mu_t)(y) > 0$. Using the representation \eqref{eq:genRFIb} and by disintegration we find 
	\begin{align*}
	\RF_Q^\lambda(\mu_t \vert \pi) &= \sum_{x,x' \in \curlX} Q(x,x') \left[\mu_t(x') \dfrac{\pi(x)}{\pi(x')} - \dfrac{1}{\lambda}\mu_t(x)^{1-\lambda} \mu_t(x')^\lambda\left(\dfrac{\pi(x)}{\pi(x')}\right)^\lambda\right] \\
	&= \sum_{y \in \curlY} \sum_{z,z' \in \curlZ} (\xi_\# \mu_t)(y) Q_y(z,z') \left[\mu_t(z' \vert y) \dfrac{\pi_y(z)}{\pi_y(z')} - \dfrac{1}{\lambda}\mu_t(z \vert y)^{1-\lambda} \mu_t(z' \vert y)^\lambda\left(\dfrac{\pi_y(z)}{\pi_y(z')}\right)^\lambda\right] \\
	&= \sum_{y \in \curlY} (\xi_\#\mu_t)(y)\RF_{Q_y}^\lambda(\mu_t(\cdot \vert y) \vert \pi_y). 
	\end{align*}
	Here, we used that the conditional measure $\pi(\cdot \vert y) \in \curlP(\curlZ)$ is the stationary measure $\pi_y$ of $Q_y$ since $(\xi_\# \pi)(y) > 0$.
	Using~\eqref{eq:locEqRelRFIvanishes} along with the the irreducibility of $Q_y$, the fact that $\pi_y \in \curlP_+(\curlZ)$ and  Lemma \ref{l:R=0} we find $\mu_t(z|y)=\pi_y(z)$ for any $(y,z)\in\curlX$ with $(\xi_\# \mu_t)(y) > 0$, and therefore~\eqref{eq:CGLimitDecomposition} follows since it holds trivially whenever $(\xi_\# \mu_t)(y) = 0$. By the convergence properties of $\xi_\#\mu^\e$ given in Lemma~\ref{lem:compactness}, we find $\hat\mu:=\xi_\#\mu\in \curlC([0,T];\curlP(\curlY))$ such that $\xi_\#\mu^\e\rightarrow \hat\mu$ uniformly in time.
\end{proof}

\subsection{Liminf inequality}\label{sec:CGLimInf}
As discussed in Section~\ref{subsec:variationalFrameworkCG}, the final step is to prove a liminf inequality which will also provide us with the limit dynamics. We prove this result in the next theorem. 

We define the (limiting) functional $\I_L:\curlC([0,T];\curlP(\curlY))\rightarrow\R$ by
\begin{equation}\label{def:LimRatFun}
\begin{aligned}
\I_L(\hat{\mu}) := \sup_{g \in \curlC^1([0,T];\Cb(\curlY))} & \left\{ \sum_{y \in \curlY} g_T(y) \hat{\mu}_T(y) - \sum_{y \in \curlY} g_0(y) \hat{\mu}_0(y) \right.\\
&- \left.\int_0^T \Biggl[\sum_{y \in \curlY} \partial_t g_t(y) \hat{\mu}_t(y) + \sum_{y,y' \in \curlY}\hat{\mu}_t(y) L(y,y') \left(e^{\nabla g_t(y',y)} - 1\right)\Biggr] \,dt \right\}, 
\end{aligned}
\end{equation}
with the (limiting) generator $L$  defined as
\begin{align}\label{def:Lamb}
L := \left(\begin{array}{cc}
- \lambda_0 & \lambda_0 \\
\lambda_1 & -\lambda_1
\end{array}\right), \quad \lambda_y := \sum_{z,z' \in \curlZ} \pi_y(z) C_{y,1-y}(z,z').
\end{align}
Here  $\pi_y\in\curlP_+(\curlZ)$ is the stationary measure of $Q_y$ (recall Lemma~\ref{lem:localEquilibrium}). 
Since $g=0$ is admissible, $\I_L\geq 0$. Furthermore we have the equivalence 
\begin{align}\label{eq:LimitEvo}
\I_L(\hat\mu)=0 \ \Longleftrightarrow \ \partial_t\hat\mu=L^T\hat\mu.
\end{align}

\begin{lem}\label{lem:liminf}
	Under the same assumptions of Lemma~\ref{lem:localEquilibrium} we assume that $\mu^\e \rightarrow \mu$ narrowly in  $\curlM([0,T] \times \curlX)$ and $\xi_{\#}\mu^\e \rightarrow \hat\mu$ in $\curlC([0,T];\curlP(\curlY))$ (recall Lemma~\ref{lem:compactness}). Then 
	\begin{align*}
	\liminf_{\e \rightarrow 0} \I_{L^\e}(\mu^\e) \geq \I_L(\hat{\mu}).
	\end{align*}
\end{lem}
\begin{proof}
	We write the rate functional $\I_{L^\e}:\curlC([0,T];\curlP(\curlX))\rightarrow\R$ (defined in~\eqref{eq:rateFunctionalTheorem}) as
	\begin{align*}
	\I_{L^\e}(\mu^\e) = \sup_{f \in \curlC^1([0,T];\Cb(\curlX))} \mathcal{J}^\e(\mu^\e, f),
	\end{align*}	
	with
	\begin{align*}
	\curlJ^\e(\mu^\e, f) := \langle f_t,\mu_T^\e\rangle - \langle f_0, \mu_0^\e\rangle - \int_0^T \sum_{x,x' \in \curlX} \mu_t (x)\left(\partial_t f_t(x) + L^\e(x,x') \left[e^{\nabla f(x',x)} -1\right]\right) \,dt.
	\end{align*}
	Using $\curlA:=\{f=g\circ\xi:\, g\in  \curlC^1([0,T];\Cb(\curlY))\}$ we have
	\begin{align*}
	\I_{L^\e}(\mu^\e) \geq \sup_{f \in \curlA} \mathcal{J}^\e(\mu^\e, f),
	\end{align*}
	where
	\begin{align}
	\mathcal{J}^\e(\mu^\e, g\circ\xi)
	&= \langle g_T\circ\xi, \mu_T^\e\rangle - \langle g_0\circ\xi,\mu_0^\e\rangle -\int_0^T \langle\partial_t(g_t\circ\xi),\mu_t^\e\rangle\,dt\nonumber\\
	&\hspace*{4em}-\int_0^T\sum_{(y,z) \in \curlY\times \curlZ } \mu^\e_t((y,z)) \sum_{z' \in \curlZ}C_{y, 1-y}(z,z')\left(e^{-\nabla g_t(y,1-y)}-1\right).\label{eq:TestFun-LimInf}
	\end{align}
	We now show that~\eqref{eq:TestFun-LimInf} converges to~\eqref{def:LimRatFun} term by term. Since $\xi_{\#} \mu_t^\e \rightarrow \hat{\mu}_t$ uniformly in $t \in [0,T]$, for the first three terms in the right hand side of~\eqref{eq:TestFun-LimInf} we find
	\[
	\langle g_T, \xi_\#\mu_T^\e\rangle - \langle g_0,\xi_\#\mu_0^\e\rangle -\int_0^T \langle\partial_tg_t,\xi_\#\mu_t^\e\rangle\,dt \quad \xrightarrow{\e \rightarrow 0} \quad\langle g_T, \xi_\#\mu_T\rangle - \langle g_0,\xi_\#\mu_0\rangle -\int_0^T \langle\partial_tg_t,\xi_\#\mu_t\rangle\,dt.
	\]
	Using Lemma~\ref{lem:localEquilibrium} for the final term in~\eqref{eq:TestFun-LimInf} yields
	\begin{align*}
	&\int_0^T  \sum_{(y,z) \in \curlY\times\curlZ} \mu^\e_t((y,z)) \left(e^{-\nabla g_t(y,1-y)}-1\right) \sum_{z' \in \curlZ} C_{y,1-y}(z,z') \,dt \\
	&\hspace*{6em}\xrightarrow{\e \rightarrow 0}\quad \int_0^T \sum_{(y,z) \in \curlY\times\curlZ} \pi_{y}(z) \hat{\mu}_t(y) \left(e^{-\nabla g(y,1-y)}-1\right) \sum_{z' \in \curlZ} C_{y,1-y}(z,z') \,dt \\
	&\hspace*{8.5em}= \int_0^T \sum_{y \in \curlY} \hat{\mu}_t(y) \left(e^{-\nabla g(y,1-y)}-1\right) \lambda_{y} \,dt.
	\end{align*} 
	where $\lambda_y$ is defined in~\eqref{def:Lamb}. Altogether, we obtain
	\begin{align*}
	\liminf_{\e \rightarrow 0}\I_{L^\e}(\mu^\e) &\geq \liminf_{\e \rightarrow 0} \mathcal{J}^\e(\mu^\e, g\circ \xi) \\
	&= \langle g_T, \hat\mu_T\rangle - \langle g_0,\hat\mu_0\rangle -\int_0^T \langle\partial_tg_t,\hat\mu_t\rangle + \sum_{y \in \curlY} \hat{\mu}_t(y) \left(e^{-\nabla g(y,1-y)}-1\right) \lambda_{y}\,dt
	\end{align*}
	for every $g\in \curlC^1([0,T];\Cb(\curlY))$. Taking the supremum over such functions concludes the proof.
\end{proof}

\begin{rem}[Limiting behaviour of solutions]\label{rem:Sol-BoundRF} 
	So far, in all the steps we have assumed that the sequence $\mu^\e$ are \emph{approximate solutions} in the sense that they satisfy $\sup_{\e >0}\I_{L^\e}(\mu^\e)<\infty$. The case when $\mu^\e$ is a sequence of solutions to the forward Kolmogorov equation~\eqref{eq:KolEq} is a special case of our analysis, which corresponds to the choice $\I_{L^\e}(\mu^\e)=0$. Lemma~\ref{lem:liminf} implies that the limiting evolution for a sequence of solutions is given by~\eqref{eq:LimitEvo}. Theorem~\ref{thm:Intro-CG-Sol} summarises the results  for a sequence of solutions.
\end{rem}

\section{Conclusion and discussion}\label{sec:conclusionDiscussion}
In this paper we construct a \emph{generalised relative Fisher Information} in the context of Markov jump processes on possibly countable discrete state space. This generalised Fisher Information has various favourable properties, and connects naturally to the relative entropy and the large deviation rate functional. We then use these connections to solve a coarse-graining problem in the context of Markov jump processes. 

We now discuss sme open questions and connected problems. 

\medskip

\textbf{Coarse-graining in more general setting.} 
As mentioned in the introduction, our coarse-graining example was already discussed using martingale techniques in \cite{lahbabiLegoll13}. Related ideas have also been discussed in~\cite[Chapter 16]{PavliotisStuart08}. We now discuss whether more general settings can also be treated by our method.
For that we distinguish two cases, finite state-spaces and countable state-spaces.
In the case of finite state-spaces, we expect that our proofs straightforwardly generalise to the case there are more than two macro-states which each have a different (finite) number of macro-states, i.e.\ $\curlY$ is an arbitrary finite set and $\curlX = \cup_{y \in \curlY} \{y\} \times \curlZ_y$.

In contrast the case of infinite state-spaces provides more difficulties.
A particular one is that the compactness argument in Lemma~\ref{lem:compactness} via Prokhorov's theorem relies on the the fact that the state-space is finite and thus compact.
In \cite{duongLamaczPeletierSharma17} this is solved by using the FIR inequality to obtain bounds on the free energy which are in turn used to obtain compactness results.
However, it is an open question, whether such a strategy is applicable in the discrete case.

\medskip

\textbf{Other stochastic processes.}
The approach to the FIR inequality presented in this work is rather general, which we now formally outline. 
Let $X$ be a smooth manifold with tangent bundle $TX$ and $\calL : X \times TX \rightarrow \R$ a Lagrangian, or more generally an $L$-function~\cite{mielkePeletierRenger14}, i.e.\ $\calL$ is nonnegative, convex in its second argument and induces an evolution equation via
\begin{align*}
\calL(x,s) = 0 \Longleftrightarrow  s = \curlA(x).
\end{align*}
Note that we do not assume that $\calL$ originates from a large deviations principle.
Furthermore, suppose that there is a smooth Lyapunov function $\calF : X \rightarrow \R$ connected to the evolution equation $\partial_t x = \curlA(x)$.

We now construct a relative entropy-type functional comparing two elements from $X$ by using the Bregman divergence of $\calF$,
\begin{align*}
\calF(x \vert y) := \calF(x) - \calF(y) - \scalarprod{d\calF(y)}{x-y},
\end{align*}
where $d\calF$ is the Fr{\'e}chet derivative of $\calF$.
Then, we can formally define the generalised relative Fisher Information in this case as
\begin{align*}
\RF^\lambda_\curlA(x \vert y) := \scalarprod{d^2 \calF(y) (\curlA(y))}{x-y} - \dfrac{1}{\lambda} \curlH(x,\lambda(d\calF(x) - d\calF(y))),
\end{align*}
where $\curlH(x,\cdot)$ is the Legendre transform of $\curlL(x,\cdot)$ for fixed $x \in X$.
By construction, these functionals satisfy the FIR-type inequality
\begin{align*}
\calF(x_T \vert y_T) - \calF(x_0 \vert y_0) + \int_0^T \RF_\curlA^\lambda(x_t \vert y_t) \,dt \leq \dfrac{1}{\lambda} \int_0^T \curlL(x_t, \partial_t x_t) \,dt,
\end{align*}
with $y :[0,T] \rightarrow X$ satisfying $\partial_t y = \curlA(y)$.
We still expect that $\RF_\curlA^\lambda$ converges for $\lambda \rightarrow 0$ to the classical relative Fisher Information $\RF_\curlA$, as motivated on page~\pageref{motivation:RFI-convergence}.
However, whether $\RF_\curlA^\lambda$ is also a non-negative functional is an open question.
We suspect that the Lagrangian and the Lyapunov function have to be connected in some appropriate sense for this to hold. 
One example of such a connection would be when both originate from a large deviations principle.

This also related to the important question, `How to construct Lypanunov functions?'.
There are, in principle, multiple approaches to do this.
For example, a specific choice can be motivated via a gradient flow result or via a large deviations principle.
In the case discussed in this work, both methods are valid.
While the fact that the relative entropy can be obtained via a large deviations principle is well known, gradient flow results for discrete state spaces are relatively new, see e.g.~\cite{maas11}. Further results for both these approaches also exist for certain nonlinear systems, see e.g.~\cite{kraaij16,erbarFathiLaschosSchlichting16}. However it is not clear if and how these are connected and whether they can be used in the construction of a generalised relative Fisher Information as described above.

\medskip

\textbf{Quantification of coarse-graining error.} The FIR inequality has been  successfully used to quantify error in relative entropy between two different forward Kolmogorov equations in the context of diffusion equations. Similar questions can be asked in the Markov jump process context, for instance to prove rates of convergence -- note that in this paper we only prove qualitative convergence. However the role of the generalised Fisher Information and the FIR inequality in proving such quantitative estimates is an open problem. To do this, we expect that the right object to consider is not the FIR inequality but a related result inspired by~\cite{yau91} (see~\cite[Chapter 8]{hilder17} for preliminary results). 

\section*{Acknowledgments}
BH was supported by the German Research Foundation (DFG) within the Cluster of Excellence in Simulation Technology (EXC 310/2) at the University of Stuttgart.
MAP and US kindly acknowledge support from the Nederlandse Organisatie voor Wetenschappelijk Onderzoek (NWO) VICI grant 639.033.008.
OT was funded by the NWO Vidi grant number 016.Vidi.189.102.

\begin{appendices}
	
	\section{Banach-space-valued functions}
	\label{app:Bochner}
	In this appendix we briefly summarize some properties of functions from an interval $[0,T]$ into the Banach space $\ell^1(\curlX)$; we follow the treatment in~\cite{HytonenVanNeervenVeraarWeis16-I} and use their terminology. While in this paper the set $\curlX$ is assumed to be either finite or countable, in this appendix we assume that $\curlX$ is countable, and to simplify notation we assume that $\curlX = \N$; the results for the finite case are all classical. 
	
	\medskip
	
	First we define the space $A\curlC([0,T];\curlP(\curlX))$ of absolutely continuous trajectories in the space of probability measures. This is the space of curves $\mu:[0,T]\to \curlP(\curlX)$ that satisfy
	\begin{quote}
		For every $\e>0$, there exists $\delta>0$ such that 
		for any finite set of disjoint intervals $([a_k,b_k])_{k\in I}\subset [0,T]$ with 
		$\sum_{k\in I} |b_k-a_k| < \delta$ we have 
		$\sum_{k\in I} \|\mu(b_k)-\mu(a_k)\|_{\ell^1(\curlX)} < \e$.
	\end{quote}
	Note that the metric used in the definition above is the $\ell^1$-norm, which is consistent because strong and weak continuity coincide. 
	
	\medskip
	
	Next we turn  to Bochner spaces. We refer to \cite{HytonenVanNeervenVeraarWeis16-I} for the concepts of measurability and Bochner integrability of a function $u:[0,T]\to\ell^1(\N)$.
	The Bochner space $L^1(0,T;\ell^1(\N))$ is defined as the space of equivalence classes of strongly Lebesgue-measurable functions with finite norm
	\[
	\|u\|_{L^1(0,T;\ell^1(\N))} := \int_0^T \|u(t)\|_{\ell^1(\N)}\, dt.
	\]
	The space $W^{1,1}(0,T;\ell^1(\N))$ is defined as the subset of $L^1(0,T;\ell^1(\N))$ of functions with weak derivatives in $L^1(0,T;\ell^1(\N))$. 		
	\begin{lem}
		\label{l:Bochner:1}
		Let $u:[0,T]\to\ell^1(\N)$; then $u\in A\curlC([0,T];\ell^1(\N))$ iff $u\in W^{1,1}(0,T;\ell(\N))$. In this case
		the derivative $\partial_t u(t)$ exists in the classical sense at almost all $t$, it is a.e.\ equal to the weak derivative of $u$, and we have
		\[
		u(\tau)-u(\sigma) = \int_\sigma^\tau \partial_tu(t)\, dt,\qquad\text{for all }0\leq \sigma\leq \tau\leq T,
		\]
		where the integral is in the sense of Bochner. 
	\end{lem}
	
	\begin{proof}
		The space $\ell^1(\N)$ is separable and is the dual of the space 
		\[
		c_0(\N)= \Bigl\{(u_n)_{n\in \N}\in \R^\N: \lim_{n\to\infty} u_n = 0 \Bigr\},
		\]
		equipped with the supremum norm. This implies that $\ell^1(\N)$ has the \emph{Radon-Nikodym property}~\cite[Th.~1.3.21]{HytonenVanNeervenVeraarWeis16-I}.
		The assertion then follows from~\cite[Th.~2.5.12 and Prop.~2.5.9]{HytonenVanNeervenVeraarWeis16-I}. 
	\end{proof}
	
	\medskip
	
	For the proof of Theorem~\ref{thm:FIR} we need a generalization of the chain rule to absolutely continuous functions with values in $\ell^1(\N)$. When $u\in A\curlC([0,T];\R)$ and $f\in \curlC^1(\R)$, the chain rule is standard and can be found e.g.\ in~\cite[Cor.~8.11]{Brezis11}; the extension to functions $f$ that are only Lipschitz is described informally in various places, such as~\cite[Remark~A.3]{ShargorodskyToland08}. The following lemma generalizes this extension to  compositions of the form $f(u(t),v(t))$ under special conditions on $f$: 
	\begin{lem}
		\label{l:two-dim-chain-rule}
		Let $A\subset \R^2$, and let $f:A\to\R$ be globally Lipschitz continuous and differentiable at each point of $A$.  Let $u,v\in A\curlC([0,T];\R)$ satisfy $(u(t),v(t))\in A$ for all $t$. Define $w(t) := f(u(t),v(t))$. Then $w$ is absolutely continuous, and the chain rule holds in the following sense. There exists a null set $N\subset [0,T]$ such that $w$, $u$, and $v$ are differentiable at each $t\in [0,T]\setminus N$, and such that 
		\begin{equation}
		\label{eq:chain-rule-R2}
		w'(t) = \partial_1 f(u(t),v(t))u'(t) +\partial_2 f(u(t),v(t))v'(t) \qquad \text{for  all }t\in[0,T]\setminus N.
		\end{equation}
	\end{lem}
	
	\begin{proof}
		First note that by the Lipschitz continuity of $f$, $w$ is absolutely continuous. To prove the chain rule~\eqref{eq:chain-rule-R2}, we restrict ourselves to the set of $t$ for which $u$, $v$, and $w$ each are differentiable; the remainder $N$ of $[0,T]$ is a null set. Consider such a $t\in [0,T]\setminus N$; since $(u(t),v(t))\in A$,  $f$ is differentiable at $(u(t),v(t))$, and therefore~\eqref{eq:chain-rule-R2} follows from the classical chain rule. 
	\end{proof}
	
	We then use the previous lemma to prove the chain rule for two nonnegative $\ell^1$-valued functions. 
	\begin{lem}
		\label{l:chain-rule-ell1}
		As in Lemma~\ref{l:two-dim-chain-rule}, let $f:A\to\R$ be globally Lipschitz continuous and differentiable at each point of $A$. Let $u,v\in A\curlC([0,T];\ell^1(\N))$ satisfy $(u(t,x),v(t,x))\in A$ for all $t$ and $x$. Define the function 
		\[
		w(t,x) = f(u(t,x),v(t,x)) \qquad \text{for each } x\in \N \text{ and } t\in [0,T].
		\]
		Then $w \in A\curlC([0,T];\ell^1(\N))$ and 
		\begin{equation}
		\label{eq:chain-rule-ell1}
		\partial_t w(t,x) = \partial_1 f(u(t,x),v(t,x))\partial_t u(t,x) + \partial_2f(u(t,x),v(t,x))\partial_t v(t,x)\qquad \text{for a.e. }t\in[0,T] \text{ and all }x\in \N.
		\end{equation}
	\end{lem}
	Note that pointwise evaluation is a 
	continuous operation on $\ell^1(\N)$, and therefore commutes with time differentiation; this shows that there is no ambiguity in the notation $\partial_t w(t,x)$, since $[w'(t)](x) = d/dt \,[w(t,x)]$ for almost all $t$ and all $x$.
	
	\begin{proof}
		The absolute continuity of $w$ follows directly from the Lipschitz continuity of $f$. 
		To prove the chain rule~\eqref{eq:chain-rule-ell1}, 
		fix $x\in \N$ and observe that $t\mapsto u(t,x)$ and $t\mapsto v(t,x)$ are elements of $A\curlC([0,T];[0,\infty))$; therefore
		\[
		\partial_t[w(t,x)] = \partial_t [f(u(t,x),v(t,x))]\  \stackrel{\mathrm{Lemma~\ref{l:two-dim-chain-rule}}}=\ \partial_1 f(u(t,x),v(t,x))\partial_t u(t,x) + \partial_2f(u(t,x),v(t,x))\partial_t v(t,x),
		\]
		for all $x$ and all $t\in [0,T]\setminus N_x$ for some null set $N_x$. Defining the null set $N:=\cup_{x\in\N} N_x$ we find that this expression holds for all $x$ and all $t\in [0,T]\setminus N$, which proves the lemma.
	\end{proof}

	\section{Proof of Theorem~\ref{thm:LDP}}\label{App-ProofLDP}
	
	The large deviation result and the corresponding rate functional 
	(see~\eqref{eq:rateFunctional}) for Markov chains on a finite or countable state space have been 
	discussed in~\cite[Proposition 5.10]{kraaij18}. The main objective of Theorem~\ref{thm:LDP} is to give 
	a different characterisation of the rate functional which is more useful in the context of coarse-
	graining (discussed in Section~\ref{subsec:variationalFrameworkCG}). The proof is inspired by  
	techniques developed in~\cite[Section 4]{dawsonGaertner87}, where the authors study large deviation 
	principles in the context of weakly-interacting diffusions. 
	
	We define
	\begin{align}
	\tilde{\LDJ}_{s,t}(\mu, f) := \sum\limits_{x\in\curlX} f_t(x)\mu_t(x) - 
	\sum\limits_{x\in\curlX} f_s(x) \mu_s(x) - \int_s^t \sum\limits_{x\in\curlX} \partial_u f_u(x) \mu_u(x) 
	+ \curlH(\mu_u, f_u) \,du.
	\label{eq:rate_functional_without_supremum}
	\end{align}
	
	\begin{cor}\label{cor:restriction_of_rate_functional}
		Let $\mu \in \curlC([0,T];\curlP(\curlX))$, $\curlI \subset \N$ a finite index set and $
		[s_k,t_k] \subset [0,T]$, $k\in\curlI$ be a finite family of pairwise disjoint intervals.
		Then for any function $g = \sum_{k \in \curlI} \varphi_k \chi_{[s_k,t_k]}\in 
		L^\infty(0,T;\Cb(\curlX))$, with $\varphi_k \in \Cb(\curlX)$ and indicator function $
		\chi_I$ (on interval $I$), there exists a monotonically decreasing sequence $g^n \in \curlC^1([0,T];
		\Cb(\curlX))$ such that $\|g^{n} - g\|_{\Cb(\curlX)}\rightarrow 0$ pointwise almost 
		everywhere in $(0,T)$ as $n \rightarrow \infty$ and
		\begin{align*}
		\tilde{\LDJ}_{0,T}(\mu, g^n)\quad \xrightarrow{n \rightarrow \infty}\quad \sum_{k \in 
			\curlI} \tilde{\LDJ}_{s_k,t_k}(\mu,g),
		\end{align*}
		where $\tilde{\LDJ}_{s,t}$ is defined by \eqref{eq:rate_functional_without_supremum}.
	\end{cor}
	\begin{proof}
		For every $k\in\curlI$ there exists a decreasing sequence $(h_{k,n})_{n \in \N} \subset 
		\curlC^1([0,T];\R)$ such that $h_{k,n}(t) \in [0,1]$ for every $t \in [0,T]$ and $h_{n,k}\rightarrow 
		\chi_{[s_k,t_k]}$ pointwise almost everywhere for $n \rightarrow \infty$.
		Furthermore, since there are only finitely many $k$ we can choose the $h_{n,k}$ such that 
		they have pairwise disjoint support for $n$ large enough.
		Finally, we assume that there exists a $C < \infty$ not depending on $n$ such that 
		\begin{align*}
		\sum_{k \in \curlI} \int_0^T \snorm{\partial_t h_{k,n}} \,dt \leq C.
		\end{align*}
		We define $g_t^n(x) := \sum_{k \in \curlI} \varphi_k(x) h_{k,n}(t) \in \curlC^1([0,T];
		\Cb(\curlX))$.
		This sequence is monotonically decreasing and satisfies $g^n \rightarrow g$ pointwise 
		almost everywhere for $n \rightarrow \infty$.
		
		Now, we recall that
		\begin{align}\label{eq:tildeJ}
		\tilde{\LDJ}_{0,T}(\mu, g^n) = \sum\limits_{x\in\curlX} g_T^n(x) \mu_T(x) - 
		\sum\limits_{x\in\curlX} g_0^n(x) \mu_0(x) - \int_0^T \sum\limits_{x\in\curlX} \partial_t g_t^n(x) 
		\mu_t(x) + \curlH(\mu_t, g^n_t) \,dt.
		\end{align}
		We first consider the asymptotic behaviour of $\int_0^T \curlH(\mu_t, g_t^n) \,dt$.
		Since $h_{k,n}$ have pairwise-disjoint support for large $n$, we find by the monotone convergence theorem
		\begin{align*}
		&\int_0^T \curlH(\mu_t, g^n_t) \,dt = \int_0^T \sum_{x,y \in \curlX} \mu_t(x) L(x,y) 
		\left[e^{\nabla g_t^n(y,x)} - 1\right] \,dt \\
		&= \sum_{k \in \curlI} \int_0^T \chi_{\operatorname{supp}(h_{k,n})} \sum_{x,y \in \curlX} 
		\mu_t(x) L(x,y) \left[e^{h_{k,n}(t)\nabla \varphi_k(y,x)} - 1\right] \,dt \\
		&\xrightarrow{n \rightarrow \infty} \sum_{k \in \curlI} \int_{s_k}^{t_k} \sum_{x,y \in 
			\curlX} \mu_t(x) L(x,y) \left[e^{\nabla \varphi_k(y,x)} - 1\right] \,dt = \sum_{k \in \curlI} 
		\int_{s_k}^{t_k} \curlH(\mu_t,g_t) \,dt.
		\end{align*}
		To study the first three terms on the right side of~\eqref{eq:tildeJ}, for any $\phi \in 
		\curlC^1([0,T];\R)$ we define 
		\begin{align*}
		\mathcal{F}_{k,n}(\phi) := h_{k,n}(T) \phi(T) - h_{k,n}(0) \phi(0) &- \int_0^T \partial_t 
		h_{k,n}(t) \phi(t) \,dt
		= \int_0^T h_{k,n}(t) \partial_t \phi(t) \,dt, \\
		\mathcal{F}_k(\phi) := \phi(t_k) - \phi(s_k) = \int_{s_k}^{t_k} \partial_t \phi(t) \,dt,
		\end{align*}
		where the second equality follows from the integration by parts formula.
		Note that both $\curlF_{k,n}$ and $\curlF_k$ are linear in $\phi$ and 
		\begin{align*}
		\snorm{\curlF_{k,n}(\phi)} &\leq \snorm{h_{k,n}(T) \phi(T)} + \snorm{h_{k,n}(0) \phi(0)} + 
		\int_0^T \snorm{\partial_t h_{k,n}(t)} \snorm{\phi(t)} \,dt \leq (2+C) \norm[\infty]{\phi}, \\
		\snorm{\curlF_k(\phi)} &\leq 2 \norm[\infty]{\phi},
		\end{align*}
		where the bounds are uniform in $n$, and that $\lim_{n \to\infty} \mathcal{F}_{k,n}(\phi) = 
		\mathcal{F}_k(\phi)$ for all $\phi\in \curlC^1([0,T];\R)$ and $k$.
		Now consider an arbitrary $\phi \in \curlC([0,T];\R)$ and sequence $\phi_l \in 
		\curlC^1([0,T];\R)$  which uniformly converges to $\phi$ for $l \rightarrow \infty$.
		Then for every $k$ we find
		\begin{align*}
		\lim_{n \rightarrow \infty} \curlF_{k,n}(\phi) = \lim_{n \rightarrow \infty} \lim_{l 
			\rightarrow \infty} \curlF_{k,n}(\phi_l) = \lim_{l \rightarrow \infty} \lim_{n \rightarrow \infty} 
		\curlF_{k,n}(\phi_l) = \lim_{l \rightarrow \infty} \curlF_k(\phi_l) = \curlF_k(\phi).
		\end{align*}
		Using this, for any $\mu \in \curlC([0,T];\curlP(\curlX))$ we find
		\begin{align*}
		&\sum_{x \in \curlX} g_T^n(x) \mu_T(x) - \sum_{x \in \curlX} g_0^n(x) \mu_0(x) - \int_0^T 
		\sum\limits_{x\in\curlX} \partial_t g_t^n(x) \mu_t(x) \,dt \\
		& =\sum_{k \in \curlI} \sum_{x \in \curlX} \varphi_k(x)\left[ h_{k,n}(T) \mu_T(x) - h_{k,n}
		(0) \mu_0(x) - \int_0^T \partial_t h_{k,n}(t) \mu_t(x)\,dt\right] \\
		& =\sum_{k\in\curlI} \sum_{x \in \curlX} \varphi_k(x) \curlF_{k,n}(\mu(x)) \xrightarrow{n 
			\rightarrow \infty} \sum_{k \in \curlI} \sum_{x \in \curlX} \varphi_k(x) \curlF_k(\mu(x)) =\sum_k 
		\left[\sum\limits_{x\in\curlX} \varphi_k(x) \mu_{t_k}(x) - \sum\limits_{x\in\curlX} \varphi_k(x) 
		\mu_{s_k}(x)\right],
		\end{align*}
		where we have used  Fubini's theorem to arrive at the first equality and the dominated 
		convergence theorem to pass to the limit. 
		Together with the convergence of the Hamiltonian proved earlier, we have the result. 
	\end{proof}
	
	\begin{proof}[Proof of Theorem~\ref{thm:LDP}]
		We first prove the large-deviation principle itself. 
		Applying~\cite{kraaij18} to the generator~$L$,  we take for its core  $D$ the space $\Cz(\curlX)$, equipped with the supremum norm, so that the dual $D'$ is isomorphic to $\ell^1$(\curlX). Then \cite[Proposition~5.10]{kraaij18} implies that $\rho^N$ satisfies a large-deviation principle in $D_{\curlP(\curlX)}[0,T]$ with rate function
		\begin{equation}
		\label{def:tilde_I_L}
		\widehat \I_L(\mu) =\begin{dcases}
		\int_0^T \widehat\curlL(\mu_t,\partial_t \mu_t) \,dt, & \text{if } \mu \in D\mhyphen A\curlC([0,T];\curlP(\curlX)), \\
		+ \infty, & \text{otherwise.}
		\end{dcases}
		\end{equation}
		Here the Lagrangian $\widehat\curlL:\curlP(\curlX) \times \ell^1(\curlX)\to[0,\infty]$ given in terms of $\curlH$ in~\eqref{eq:hamiltonian} by
		\[
		\widehat\curlL(\mu,s) := \sup_{f\in \Cz(\curlX)}\  \langle f,s\rangle - \curlH(\mu,f),
		\]
		and the space $D\mhyphen A\curlC([0,T];\curlP(\curlX))$ is the space of  curves $\nu:[0,\infty)\to\curlP(\curlX)$ such that $t\mapsto \langle f,\nu(t)\rangle$ is absolutely continuous for all $f\in D=\Cz(\curlX)$, with a unique weak-star measurable derivative $u:[0,\infty)\to D'=\ell^1(\curlX)$ in the sense that $(d/dt) \langle \nu(t),f\rangle = \langle f,u(t)\rangle$ for all $f\in \Cz(\curlX)$ and $t\geq0$.
		
		The rate function $\widehat \I_L$ in~\eqref{def:tilde_I_L} differs from $\I_L$ in~\eqref{eq:rateFunctional} in two ways. First, the explicit domain of definition in~\eqref{eq:rateFunctional} is $A\curlC([0,T];\curlP(\curlX))$, the space of curves that are absolutely continuous in $\ell^1(\curlX)$; this is a subspace of $D\mhyphen A\curlC([0,T];\curlP(\curlX))$. Secondly, $\curlL(\mu,s)$ is defined as a supremum over $\ell^\infty(\curlX)$, while $\widehat\curlL(\mu,s)$ is defined as the same supremum but over the smaller space $\Cz(\curlX)$, implying that $\widehat \curlL\leq \curlL$. 

		Nonetheless, we have $\widehat \I_L = \I_L$. 
		To show this, we first note that for $s\in \ell^1(\curlX)$ and $\mu\in \curlP(\curlX)$, we have
		\begin{equation}
		\label{eq:widehatL-density}
		\sup_{f\in \ell^\infty(\curlX)}\  \langle f,s\rangle - \curlH(\mu,f)= 
		\sup_{f\in c_0(\curlX)}\  \langle f,s\rangle - \curlH(\mu,f),
		\end{equation}
		and therefore $\widehat\curlL(\mu,s) = \curlL(\mu,s)$ for all $s\in \ell^1(\curlX)$.
		Indeed, fix $s\in \ell^1(\curlX)$ and $f\in \ell^\infty(\curlX)$, and let $f_n\in c_0(\curlX)$ be the truncation of $f$ to the first $n$ elements of $\curlX$. Then 
		\begin{align*}
		\sum_{x\in\curlX} f_n(x)s(x) &\to \sum_{x\in\curlX} f(x)s(x) 
		\qquad\text{and}\qquad\\
		\sum_{x,y \in \curlX} \mu(x) L(x,y) \left[e^{f_n(y) - f_n(x)} - 1\right]
		&\to \sum_{x,y \in \curlX} \mu(x) L(x,y) \left[e^{f(y) - f(x)} - 1\right],
		\end{align*}
		both by the dominated convergence theorem, since $s\in\ell^1(\curlX)$ and $(x,y)\mapsto \mu(x)L(x,y)\in \ell^1(\curlX\times\curlX)$. This proves~\eqref{eq:widehatL-density}, and shows that for $s\in \ell^1(\curlX)$, $\widehat\curlL(\mu,s) = \curlL(\mu,s)$.
		
		Next, by \cite[Proposition~2.12]{kraaij18}, curves $\mu$ with $\widehat \I_L(\mu)<\infty$ satisfy $\mu\in A\curlC([0,T];\curlP(\curlX))$. Since curves in $A\curlC([0,T];\curlP(\curlX))$ have derivatives in $\ell^1$, any curve with $\widehat \I_L(\mu)<\infty$ satisfies 
		\[
		\widehat \I_L(\mu) = \int_0^T \widehat\curlL(\mu_t,\partial_t \mu_t) \,dt
		= \int_0^T \curlL(\mu_t,\partial_t \mu_t) \,dt = \I_L(\mu).
		\]
		This proves that $\widehat \I_L = \I_L$ whenever $\widehat \I_L<\infty$. For the remaining case $\widehat\I_L(\mu)=\infty$ there are three possibilities:
		\begin{enumerate}
			\item $\mu\not\in D\mhyphen A\curlC([0,T];\curlP(\curlX))$, therefore $\mu\not\in A\curlC([0,T];\curlP(\curlX))$ and $\I_L(\mu)=\infty$ also; 
			\item $\mu\in D\mhyphen A\curlC([0,T];\curlP(\curlX))$ but $\mu\not\in A\curlC([0,T];\curlP(\curlX))$ and again $\I_L(\mu)=\infty$; 
			\item $\mu\in A\curlC([0,T];\curlP(\curlX))$ but
			\[
			\infty = \int_0^T \widehat\curlL(\mu_t,\partial_t \mu_t) \,dt
			\leq \int_0^T \curlL(\mu_t,\partial_t \mu_t) \,dt,
			\]
			so that again $\I_L(\mu)=\infty$.
		\end{enumerate}
		This proves that $\I_L = \widehat \I_L$ and concludes the proof of the large-deviation principle.
		
		\bigskip
		
		We now continue with the characterization~\eqref{eq:rateFunctionalTheorem}.
		We define
		\begin{align*}
		\tilde{\I}_L(\mu) := \sup_{f \in \curlC^1([0,T]; \Cb(\curlX))} \tilde{\LDJ}_{0,T}(\mu, 
		f),
		\end{align*}
		where $\tilde{\LDJ}_{0,T}$ is given by \eqref{eq:rate_functional_without_supremum}.
		
		The plan of the proof is now as follows.
		We first show that $\tilde{\I}(\mu) < \infty$ for $\mu \in \curlC([0,T];\curlP(\curlX))$ 
		implies that $\mu \in A\curlC([0,T];\curlP(\curlX))$.
		We then show that $\I_L(\mu) \geq \tilde{\I}_L(\mu)$ and vice versa which yields the 
		equality.
		In particular, applying integration by parts in \eqref{eq:rate_functional_without_supremum} 
		since $\mu \in A\curlC([0,T];\curlP(\curlX))$, yields 
		\[
		\I_L(\mu) = \sup_{f \in L^\infty(0,T;\Cb(\curlX))} \int_0^T \scalarprod{f_t}
		{\partial_t \mu_t} - \curlH(\mu_t, f_t) \,dt,
		\]
		which is the last part of the statement.
		
		We now show by contradiction that $\mu \in \curlC([0,T];\curlP(\curlX))$ and $\tilde{\I}
		_L(\mu) < \infty$ implies $\mu \in A\curlC([0,T];\curlP(\curlX))$. Suppose $\tilde{\I}_L(\mu) < 
		\infty$, but $\mu \notin A\curlC([0,T];\curlP(\curlX))$, i.e.\ there exists an $\e > 0$ such that for 
		any $\delta > 0$,  there exists a finite family of pairwise-disjoint intervals $[s_k,t_k] \subset [0,T]
		$, $k \in \curlI$ with
		\begin{align*}
		\sum_{k \in \curlI} \snorm{t_k - s_k} < \delta \quad\text{ and }\quad \sum_{k \in \curlI} 
		\sum_{x \in \curlX} |\mu_{t_k}(x) - \mu_{s_k}(x)| \geq \e.
		\end{align*}
		Next, for an arbitrary $A > 0$, we define $g\in L^\infty(0,T;\Cb(\curlX))$ as
		\begin{align*}
		g_t(x) := A \sum_{k \in \curlI} \operatorname{sign}(\mu_{t_k}(x) - \mu_{s_k}(x)) 
		\chi_{[s_k,t_k]}(t).
		\end{align*}
		Using Corollary~\ref{cor:restriction_of_rate_functional}, there exists a sequence $g^n \in 
		\curlC^1([0,T];\Cb(\curlX))$ such that
		\begin{align}
		\tilde{\LDJ}_{0,T}(\mu, g^n)\quad \xrightarrow{n \rightarrow \infty}\quad \sum_{k \in 
			\curlI} \tilde{\LDJ}_{s_k,t_k}(\mu, g).
		\label{eq:restriction_of_rate_functional}
		\end{align}
		Note that the latter expression is well defined since $g\vert_{[s_k,t_k]} \in 
		\curlC^1([s_k,t_k];\Cb(\curlX))$ for all $k \in \curlI$.
		Moreover, there exists a $C < \infty$ which only depends on $\mu$ and $L$ such that
		\begin{align*}
		\sum_{k \in \curlI} \int_{s_k}^{t_k} \curlH(\mu_t, g_t) \,dt \leq C e^A \sum_{k \in \curlI} 
		\snorm{t_k - s_k} < C e^A \delta,
		\end{align*}
		since $\operatorname{sign}(\mu_{t_k} - \mu_{s_k})$ is uniformly bounded in $\curlX$.
		Furthermore, we find
		\begin{align*}
		\sum_{k \in \curlI} \left[\sum_{x \in \curlX} g_{t_k}(x) \mu_{t_k}(x) - \sum_{x \in \curlX} 
		g_{s_k}(x) \mu_{s_k}(x)\right] &= A \sum_{k \in \curlI} \sum_{x \in \curlX} \operatorname{sign}
		(\mu_{t_k}(x) - \mu_{s_k}(x))(\mu_{t_k}(x) - \mu_{s_k}(x)) \\
		&= A \sum_{k \in \curlI} \sum_{x \in \curlX} |\mu_{t_k}(x) - \mu_{s_k}(x)| \geq A\e.
		\end{align*}
		Thus, using \eqref{eq:restriction_of_rate_functional} we find 
		\begin{align*}
		\tilde{\LDJ}_{0,T}(\mu,g^n) \geq \frac{1}{2} \sum_{k \in \curlI} \tilde{\LDJ}_{s_k,t_k}
		(\mu,g) \geq\frac{1}{2} \left(A\e - C e^A \delta\right),
		\end{align*}
		for sufficiently large $n$. Since $\delta > 0$ and $A > 0$ were arbitrary, the right-hand 
		side can be arbitrarily large. More specifically, for a given $A$, we choose $\delta=\e Ae^{-A}/(2C)$, 
		thereby yielding
		\begin{align*}
		\tilde{\I}_L(\mu) \geq \tilde{\LDJ}_{0,T}(\mu, g^n) \geq \dfrac{1}{4}\e A.
		\end{align*}
		Since $A$ can be made arbitrarily large, this contradicts $\tilde{\I}_L(\mu) < \infty$. Hence, $\mu \in 
		\curlC([0,T];\curlP(\curlX))$ and $\tilde{\I}_L(\mu) < \infty$ imply that $\mu \in A\curlC(0,T;
		\ell^1(\curlX))$.
		
		Next, we show that $\I_L(\mu) \geq \tilde{\I}_L(\mu)$. 
		For $\mu \notin A\curlC([0,T];\curlP(\curlX))$ we have $\I_L(\mu) = \infty$ and therefore $\I_L(\mu) \geq \tilde{\I}_L(\mu)$.
		For $\mu \in A\curlC([0,T];\curlP(\curlX))$, on the other hand, we have
		\begin{align*}
		\I_L(\mu) = \int_0^T \curlL(\mu_t, \partial_t \mu_t) \geq \int_0^T \scalarprod{f_t}
		{\partial_t \mu_t} - \curlH(\mu_t, f_t) \,dt = \tilde{\LDJ}_{0,T}(\mu,f),
		\end{align*}
		for any curve $f \in \curlC^1([0,T];\Cb(\curlX))$, where we used integration by parts 
		to arrive at the final equality. This yields $\I_L(\mu) \geq \tilde{\I}_L(\mu)$.
		
		We complete  the proof by showing that $\I_L(\mu) \leq \tilde{\I}_L(\mu)$ for  $\mu \in A\curlC([0,T];\curlP(\curlX))$.
		Note that since $\mu \in A\curlC([0,T];\ell^1(\curlX))\equiv W^{1,1}(0,T;\ell^1(\curlX))
		$ and
		\[
		\tilde{\LDJ}_{0,T}(\mu,f) = \int_0^T \scalarprod{f_t}{\partial_t \mu_t} - \curlH(\mu_t, 
		f_t) \,dt\qquad \text{for all\, $f\in \curlC^1([0,T];\Cb(\curlX))$},
		\]
		we have that $\tilde{\LDJ}(\mu, \cdot) : L^\infty(0,T;\Cb(\curlX)) \rightarrow \R$ 
		is a continuous (nonlinear) functional.
		Since every element in $L^\infty(0,T;\Cb(\curlX))$ can be  approximated pointwise by 
		a sequence in $\curlC^1([0,T];\Cb(\curlX))$, we obtain with the dominated convergence theorem that
		\begin{align*}
		\sup_{f \in \curlC^1([0,T];\Cb(\curlX))} \tilde{\LDJ}(\mu, f) = \sup_{f \in 
			L^\infty(0,T; \Cb(\curlX))} \tilde{\LDJ}(\mu,f).
		\end{align*}
		Now, for any fixed $\e > 0$ and for almost every $t \in (0,T)$ exists a $g_t \in 
		\Cb(\curlX)$ such that
		\begin{align*}
		\sum\limits_{x\in\curlX} g_t(x) \partial_t \mu_t(x) - \curlH(\mu_t, g_t) \geq \max\left\{\curlL(\mu_t, \partial_t \mu_t) - \e, 0\right\},
		\end{align*}
		where we used the definition of the Lagrangian.
		Note that $t \mapsto g_t$ might not be an element of $L^\infty(0,T;\Cb(\curlX))$.
		Therefore, we define the sequence
		\begin{align*}
		f_{t}^k(x) := \begin{dcases}
		g_t(x) & \text{if } \norm[\Cb(\curlX)]{g_t} \leq k, \\
		0 & \text{otherwise,}
		\end{dcases}
		\end{align*}
		with $k \in \N$.
		Then, by construction we have that $0 \leq \sum_{x\in\curlX} f_{t}^k(x) \partial_t\mu_t(x) 
		- \curlH(\mu_t, f_{t}^k) \leq \curlL(\mu_t, \partial_t \mu_t)$ for all $k\in\N$, where $t \mapsto 
		\curlL(\mu_t, \partial_t \mu_t) \in L^1(0,T;[0,\infty))$ since $\mu \in A\curlC([0,T];\curlP(\curlX))
		$. Furthermore, using $f_t^k(x) \leq g_t(x)$ for all $x \in \curlX$ and almost all $t \in (0,T)$ and 
		the dominated convergence theorem, we find
		\[
		\sum\limits_{x\in\curlX} f_{t}^k(x) \partial_t\mu_t(x) - \curlH(\mu_t, f_{t}^k)\quad 
		\xrightarrow{k \rightarrow \infty}\quad \sum\limits_{x\in\curlX} g_t(x) \partial_t \mu_t(x) - 
		\curlH(\mu_t, g_t),
		\]
		for almost every $t\in(0,T)$. Hence, we can apply the dominated convergence theorem to obtain
		\begin{align*}
		\lim_{k \rightarrow \infty} \tilde{\LDJ}(\mu, f^k) = \tilde{\LDJ}(\mu, g) \geq \int_0^T 
		\curlL(\mu_t, \partial_t \mu_t) \,dt - \e T.
		\end{align*}
		Finally, since $f^k \in L^\infty(0,T;\Cb(\curlX))$ for all $k \in \N$ we obtain that 
		the left-hand side is bounded from above by $\tilde{\I}_L(\mu)$.
		Therefore, since $\e > 0$ was arbitrary, we obtain $\tilde{\I}_L(\mu) \geq \I_L(\mu)$ which 
		proves the statement.
	\end{proof}
	
	\section{Positivity of solution to the forward Kolmogorov equation}	\label{app:positivity}
	In this appendix we show that the solution to the forward Kolmogorov equation with a bounded and irreducible generator is strictly positive. While we expect this result to be true, we could not find a reference for it, and therefore provide the result here for completeness. 
	
	\begin{lem}\label{lem:Pos-fKol}
		Let $\rho\in AC([0,T];\curlP(\curlX))$ be a solution to~\eqref{eq:forwardKolmogorovEquation},
		where the generator $L$ satisfies~\eqref{eq:conditionsGenerator1}-\eqref{eq:conditionsGenerator3}. Then $\rho_t\in\curlP_+(\curlX)$ for every $t>0$. 
	\end{lem}
	\begin{proof}
		Since $L$ is a bounded Markov generator with $\nu:=\sup_{x\in\curlX} |L(x,x)|<\infty$, we can write $L=P-\nu I$ for a matrix $P$ with non-negative entries and identity matrix $I$. Since $L^T$ generates  a uniformly continuous semigroup on $\ell^1(\curlX)$ which conserves mass, we can write $e^{tL^T}=\sum_{n\geq 0} \frac{t^n (L^T)^n}{n!}(x,y)=e^{t(P^T-\nu I)}=e^{-\nu t} e^{t P^T}$. We will show that $e^{tL^T}(x,y)>0$, by proving that  $e^{tP^T}(x,y)>0$.
		
		Since $L$ is irreducible, for every $x,y\in\curlX$ with $x\neq y$, there exists a finite sequence $x_0,x_1,\ldots,x_N\in\curlX$ containing no doubled points with $x_0=x,\ x_N=y$ and $L(x_{n},x_{n+1})>0$. Using $L=P-\nu I$, $P(x_n,x_{n+1})= L(x_n,x_{n+1})>0$ we find
		\begin{align*}
		(P^T)^N(x,y)\geq \sum\limits_{i=1}^{N-1}P^T(x_{i+1},x_i)P^T(x_{i},x_{i-1})= \sum\limits_{i=1}^{N-1}P(x_i,x_{i+1})P(x_{i-1},x_{i})>0. 
		\end{align*}
		Therefore 
		\begin{align*}
		e^{tP^T}(x,y)\geq \frac{t^N(P^T)^N}{N!}(x,y)>0.
		\end{align*}
		Since $x,y\in\curlX$ are arbitrary, it follows that $e^{tL^T}$ is a positive semigroup and therefore $e^{tL^T}:\curlP(\curlX)\rightarrow\curlP_+(\curlX)$ for all $t > 0$.
	\end{proof}

	\end{appendices}
	\phantomsection
	
	\addcontentsline{toc}{section}{References}
	\bibliography{BibDeskLibrary.bib}
	\bibliographystyle{../alphainitials}
	
	\vspace{0.5cm}

	(B.\ Hilder) Institut f{\"u}r Analysis, Dynamik und Modellierung, Universit{\"a}t Stuttgart, Pfaffenwaldring 57, 70569 Stuttgart, Germany\\
	E-mail address:
	\href{mailto:bastian.hilder@mathematik.uni-stuttgart.de}{bastian.hilder@mathematik.uni-stuttgart.de}
	
	(M.\ A.\ Peletier)  Department of Mathematics and Computer Science and Institute for Complex Molecular
	Systems, Eindhoven University of Technology, 5600 MB Eindhoven, The Netherlands\\
	E-mail address: \href{mailto:M.A.Peletier@tue.nl}{M.A.Peletier@tue.nl}
	
	(U.\ Sharma) CERMICS, {\'E}cole des Ponts ParisTech, 6-8 Avenue Blaise Pascal, Cit{\'e} Descartes, Marne-la-Vall{\'e}e, 77455, France\\
	E-mail address:
	\href{mailto:upanshu.sharma@enpc.fr}{upanshu.sharma@enpc.fr}
	
	(O.\ Tse) Department of Mathematics and Computer Science, Eindhoven University of Technology, 5600 MB Eindhoven, The Netherlands\\
	E-mail address:
	\href{mailto:: o.t.c.tse@tue.nl}{o.t.c.tse@tue.nl}

\end{document}